\documentclass[12pt]{amsart}
\usepackage {enumerate}
\usepackage{setspace}
\usepackage{caption}
\usepackage{mathrsfs}
\usepackage{hyperref}
\usepackage{esint}
\usepackage{amssymb}
\usepackage{graphicx}
\usepackage{epstopdf}
\usepackage{amsthm}
\usepackage{appendix}
\usepackage{color}
\usepackage{lipsum}
\usepackage{seqsplit}
\renewcommand{\seqinsert}{\ifmmode\allowbreak\else\-\fi}

\newtheorem{theorem}{Theorem}[section]
\newtheorem{lemma}[theorem]{Lemma}

\newtheorem{proposition}[theorem]{Proposition}
\newtheorem{example}[theorem]{Example}

\newtheorem{remark}[theorem]{Remark}
\numberwithin{equation}{section}

\theoremstyle {definition}

\makeatletter
\@namedef{subjclassname@2020}{%
  \textup{2020} Mathematics Subject Classification}
\makeatother

\renewcommand{\L}{\mathcal{L}}

\DeclareMathOperator{\id}{Id}

\DeclareMathOperator{\R}{\mathbb R}

\renewcommand{\d}{\mathrm d}

\renewcommand{\S}{\ensuremath{\mathbb{S}}}

\begin{document}
\title[Singular Yamabe problem]{Solutions of the singular Yamabe problem near singular boundaries}

\begin{abstract}
In this paper, we investigate the asymptotic behaviors of solutions to the singular Yamabe problem with negative constant scalar curvature near singular boundaries and  derive optimal estimates, where the background metrics are not assumed to be conformally flat. Specifically, we demonstrate that for a wide class of Lipschitz domains with asymptotic conical structure, the local positive solutions are well approximated by the positive solutions in the tangent cones at singular boundary points. This extends the results of \cite{HJS, HS,SW}. 
\end{abstract}

\author[Shen]{Weiming Shen}
\address{School of Mathematical Sciences, Capital Normal University, Beijing, 100048, China}
    \email{wmshen@aliyun.com} 
    
\author[Wang]{Zhehui Wang}
\address{School of Sciences, Great Bay University, Dongguan, 523000, China}
    \email{wangzhehui@gbu.edu.cn}
    
\author[Xie]{Jiongduo Xie}
\address{School of Mathematical Sciences\\
Shanghai Jiao Tong University\\
Shanghai, 200240, China}
\email{jiongduoxie@outlook.com
}

\maketitle
\section{Introduction}\label{Intro}
{ 
The singular Yamabe problem concerns finding a complete conformal metric of constant scalar curvature on a domain obtained by removing a nonempty closed subset from a compact Riemannian manifold. 
In this paper, we focus on the asymptotic behaviors of solutions in the case that the resulting complete conformal metric has constant negative scalar curvature. 

Let $(M, g)$ be a smooth Riemannian manifold  of dimension $n\geq 3$, which is either closed or non-compact and complete. Let $\Omega\subset M$ be a domain, and denote by $S_g$ the scalar curvature of $g$.} Suppose $u>0$ solves
\begin{align}
    \Delta_g u-\frac{n-2}{4(n-1)}S_gu&=\frac{n(n-2)}{4}u^{\frac{n+2}{n-2}} \mbox{ in } \Omega, \label{equation-Yamabe-1}\\
    u&=+\infty \mbox{ on } \partial\Omega,\label{boundary-condition-1}
\end{align}
then $u^{\frac{4}{n-2}}g$ is a complete metric with constant scalar curvature $-n(n-1)$ in $\Omega$. 
The existence theory for problem \eqref{equation-Yamabe-1}–\eqref{boundary-condition-1} was established in the works such as \cite{AMC, LG, LN, SW}. In the case where the background manifold is the standard Euclidean space, \eqref{equation-Yamabe-1}–\eqref{boundary-condition-1} is reduced to the Loewner-Nirenberg problem, namely,
\begin{align}
    \Delta u&=\frac{n(n-2)}{4}u^{\frac{n+2}{n-2}} \mbox{ in } \Omega, \label{equation-LN-1}\\
    u&=+\infty \mbox{ on } \partial\Omega.\label{boundary-condition-LN}
\end{align}
Let $V\subset\R^n$ be an infinite cone with the origin being its vertex, and we assume $\Sigma=V\cap\S^{n-1}$ is a Lipschitz domain on $\S^{n-1}$. {If $(M, g)=(\R^n, g_0)$ and $\Omega=V\subset \R^n$, where $g_0$ is the Euclidean metric, then there is a unique positive solution $u_V$ to \eqref{equation-LN-1}-\eqref{boundary-condition-LN}. From \cite[Theorem 2.3]{HJS}, $u_V$ has the form $$u_V(x)=|x|^{-\frac{n-2}{2}}\xi(\theta),$$
where $\xi$ satisfies
\begin{align}
    \Delta_\theta \xi-\frac{(n-2)^2}{4}\xi&=\frac{n(n-2)}{4}\xi^{\frac{n+2}{n-2}} \mbox{ in } \Sigma, \label{equation-Yamabe-2}\\
    \xi&=+\infty \mbox{ on } \partial\Sigma,\label{boundary-condition-2}
\end{align}
$\theta=\frac{x}{|x|}\in\Sigma$ and $\Delta_\theta$ is the Laplace-Beltrami operator on $\S^{n-1}$.}

When $\Omega$ is a sufficiently smooth domain, the asymptotic behaviors of positive solutions to \eqref{equation-Yamabe-1}–\eqref{boundary-condition-1} near the boundary has been thoroughly studied. For bounded $C^2$ domains $\Omega$, let $d$ denote the distance function to $\partial\Omega$ and let $u$ be a positive solution of \eqref{equation-LN-1}-\eqref{boundary-condition-LN}, Loewner and Nirenberg \cite{LN} showed that near the boundary,
\begin{equation}\label{LN-0order}
|d^{\frac{n-2}{2}}u-1|\leq Cd,
\end{equation}
where $C$ is a positive constant depending only on $n$ and $\Omega$. When {$\partial\Omega$ is smooth}, Mazzeo \cite{M} and Andersson, Chru\'sciel, and Friedrich \cite{ACF} established asymptotic expansions of solutions of \eqref{equation-Yamabe-1}–\eqref{boundary-condition-1} up to an arbitrarily finite order. The singular Yamabe problem has found numerous geometric applications in recent years, {see \cite{CMY2022, CLW2019, Graham2017, GG2021,GH2017} for example.} In these works, the boundary behaviors of solutions to \eqref{equation-Yamabe-1}–\eqref{boundary-condition-1} play a crucial role.

Recently, the asymptotic behaviors of solutions {\eqref{equation-Yamabe-1}-\eqref{boundary-condition-1}} near singular boundaries have attracted increasing interest. When the background metric is flat, Han and Shen \cite{HS} proved that near a singular boundary point, the solution can be approximated by the corresponding solution in the tangent cone at that point. Besides, Jiang \cite{Jiang} studied boundary expansion of solutions in finite cones over smooth spherical domains. Furthermore, Han, Jiang, and Shen \cite{HJS} derived arbitrarily finite order asymptotic expansions of solutions in finite cones in terms of the corresponding solutions in infinite cones. For non-flat background metrics, Shen and Wang \cite{SW} extended the results of \cite{HS}, where {they applied a new method to construct barrier functions.} In both \cite{HS} and \cite{SW}, it is assumed that the boundary near a singular point consists of $k$ $(k\leq n)$ $C^2$-hypersurfaces intersecting at that point, with linearly independent normal vectors at that point. However, the asymptotic behaviors of solutions near more general types of singular points remain unknown. {Regarding the asymptotic behaviors of solutions of the singular Yamabe problem with positive constant scalar curvature, relevant results can be found in \cite{CGS1989, HLL2021, KMPS1999} (flat background metric) and in \cite{HXZ2023, M2008,XZ2022} (non-flat background metric).}

\subsection{Assumption}\label{assum}
In this paper, we will investigate the asymptotic behaviors of local positive solutions of \eqref{equation-Yamabe-1}-\eqref{boundary-condition-1} near singular points on {the boundary}. The background metric does not need to be conformally flat, and we assume for convenience that $g$ is a Riemannian metric in $B_2(0)\subset\R^n$ with $n\geq 3$, and $\Omega\subset B_2(0)$ is a domain with $0\in \partial\Omega$. Suppose $(x_1, \cdots, x_n)$ is a normal coordinate system of $g$ with $0$ as the origin and $\partial\Omega\cap\overline{B_1(0)}$ is Lipschitz.

Assume $V$ is a tangent cone of $\Omega$ at $0$, and there is a constant $0<R<1$, and a $C^2$-diffeomorphism $$T: B_R(0)\to T(B_R(0))\subset\R^n,$$
such that 
\begin{align}
T(\Omega\cap B_R(0))=V\cap T(B_R(0)),&\mbox{ and } T(\partial\Omega\cap B_R(0))=\partial V\cap T(B_R(0)),\label{T-map}\\
T(0)=0,&\mbox{ and } \nabla T(0)=\id.\label{T(0)-nablaT(0)}
\end{align}
The following presents a type of singular boundary that differs from that in references \cite{HS,SW}.

\begin{example}
Let $V_0\subsetneq\{y\in\R^n:y_n\geq0\}$ be an infinite cone with the origin being its vertex, and we assume that it has rotational symmetry. Suppose $S_0: V_0\to \R^n$ is a diffeomorphism near $0$ defined by $$S_0y=y+c|y|^2e_n,$$ where $c>0$ is a small constant and $e_n=(0,\cdots,0,1)$. Set {$T_0=S_0^{-1}$, $x=S_0y$, and $y=T_0x$} for $y\in V_0$ near $0$. By a direct calculation,
$$T_0x=(x_1,\cdots,x_{n-1},\frac{-1+\sqrt{1+4c(x_n-c|x'|^2)}}{2c}), \quad \mbox{ for } x \mbox{ near } 0,$$
where $x'=(x_1,\cdots,x_{n-1},0)$. Then, $T_0(0)=0$, $\nabla T_0(0)=\id$, and $\|T\|_{C^2}$ is bounded near $0$. Let $\Omega_0=S_0(V_0)$, then $\Omega_0$ is a domain with conical singularity at $0$, and $V_0$ is the tangent cone of $\Omega_0$ at $0$. Moreover, near $0$, $\partial\Omega_0$ does not consist of $k$ $C^2$-hypersurfaces intersecting at $0$ for any $k\leq n$.
\end{example}
\subsection{Main results}

Suppose $u>0$ in $\Omega\cap B_1(0)$ solves
\begin{align}
    \Delta_g u-\frac{n-2}{4(n-1)}S_gu&=\frac{n(n-2)}{4}u^{\frac{n+2}{n-2}} \mbox{ in } \Omega\cap B_1(0), \label{equation-Yamabe-main}\\
    u&=+\infty \mbox{ on } \partial\Omega\cap B_1(0),\label{boundary-condition-main}
\end{align}
where $S_g$ is the scalar curvature of $g$. We call such $u$ a local positive solution, and in fact, we can compare it with the unique positive solution in the tangent cone $V$ via the diffeomorphism $T$ near the {singular} point $0$, namely, compare $u(x)$ and $u_V(Tx)$ when $d_g(x, 0)\to 0$. Our first main result is as follows.

\begin{theorem}\label{thm-main-1}
Let $n\geq 3$, and let $\Omega$, $V$, and $T$ be as in assumption \ref{assum}. Suppose $u\in C^\infty(\Omega\cap B_1(0))$ is a local positive solution to \eqref{equation-Yamabe-main}-\eqref{boundary-condition-main}. Then, there is a constant $r_0>0$ small enough, such that
    \begin{equation}\label{main-1}\left|\frac{u(x)}{u_V(Tx)}-1\right|\leq Cd_g(x, 0), \mbox{ for any } x\in\Omega\cap B_{r_0}(0),\end{equation}
where $d_g(x, 0)$ is the distance from $x$ to $0$ with respect to the metric $g$, and $C$ is a positive constant depending only on $n$, $g$, $\|T\|_{C^2(B_{R}(0))}$, and $V$.

\end{theorem}
{Theorem \ref{thm-main-1} could be viewed as a generalization of the main results in \cite{HS,SW}.} Specifically, let $\partial\Omega\cap B_1(0)$ consist of $k$ $(k\leq n)$ $C^2$-hypersurfaces $S_1,\cdots,S_k$ intersecting at $0$ with the property that the
normal vectors of $S_1,\cdots,S_k$ at $0$ are linearly independent. Then, the corresponding diffeomorphism $T$ can be explicitly constructed and satisfies \eqref{T-map}-\eqref{T(0)-nablaT(0)}, {see \cite[Section 2]{SW}. Actually, estimate \eqref{main-1} is similar to \eqref{LN-0order} in some extent. When the background manifold is the standard Euclidean space and $\partial\Omega$ is $C^2$, the tangent cone at any boundary point is a half space, and for convenience we may assume the boundary point is $0$ and the tangent cone is $\mathbb{R}^n_+$ in this case. 
For $x\in\Omega$ with $|x|$ small enough and $d(x, \partial\Omega)=d(x,0)$, estimate \eqref{main-1} at $x$ is similar to estimate \eqref{LN-0order} at $x$. From this perspective, the type of estimate in Theorem \ref{thm-main-1} is reasonable.}


{When $T$ has higher regularity, we can} investigate more refined asymptotic expansions of the solutions.
\begin{theorem}\label{thm-main-2}
Under the same assumptions as in Theorem \ref{thm-main-1}, we further assume that $T\in C^3(B_R(0))$.
Then, there is a constant $r_1>0$ small enough and a function $c_1\in C^{\infty}(\Sigma)\cap L^\infty(\Sigma)$, such that 
\begin{equation}\label{main-2}
\left|\frac{u(x)}{u_V(Tx)}-1-c_1\left(\frac{Tx}{|Tx|}\right)|Tx|\right|\leq
\begin{cases}
\begin{aligned}
 &Cd^2_g(x, 0),\text{ if }\mu_1>2,\\  
 &Cd^2_g(x, 0)|\log{d_g(x, 0)}|,\text{ if }\mu_1=2,\\
 &Cd^{\mu_1}_g(x, 0),\text{ if }\mu_1<2,
\end{aligned}
\end{cases} 
\end{equation}
for any $x\in\Omega\cap B_{r_1}(0)$, where $c_1$ depends only on $n$, $g$, $V$, and the derivatives of $T$ at $0$, 
\begin{equation}\label{mu1-def}\mu_1=\sqrt{\Big(\frac{n-2}{2}\Big)^2+\lambda_1},\end{equation}
with $\lambda_1$ represents the first eigenvalue of the elliptic operator $L_1$ defined by \eqref{L1} depending only on $n$ and $\Sigma$, and $C$ is a positive constant depending only on $n$, $g$, $\|T\|_{C^3(B_R(0))}$, and $V$. {In addition}, when $\mu_1=2$ and $T\in C^{3,\gamma}(B_R(0))$ for some $\gamma\in(0, 1)$, we can improve the estimate as follows:
\begin{equation}\label{main-2-3}\left|\frac{u(x)}{u_V(Tx)}-1-c_1\left(\frac{Tx}{|Tx|}\right)|Tx|-c_2\left(\frac{Tx}{|Tx|}\right)|Tx|^2\log|Tx|\right|\leq Cd^2_g(x, 0),\end{equation} where $c_2\in C^{\infty}(\Sigma)\cap L^\infty(\Sigma)$ depends only on $n$, $g$, $V$, and the derivatives of $T$ at $0$.
\end{theorem}

\begin{remark}\label{mu1g1}
By Shen and Wang \cite[Proposition 4.5]{SW}, when $n=3$ we know $\lambda_1>\frac{3}{4}$. Therefore, we have $\mu_1>1$. 

It should be mentioned here that, by formal calculations, we can not determine the coefficient of the term with order $r^{\mu_1}$. See Remark \ref{premu1} for a detailed discussion. 
\end{remark}


\begin{remark}\label{tmalpha}
In fact, to obtain the term $c_1(Tx/|Tx|)|Tx|$ in \eqref{main-2}, it suffices to assume that $T\in C^{2,\alpha}(B_R(0))$. In general, let $m$ be an integer with {$1\leq m\leq\mu_1$}. Under the same assumptions as in Theorem \ref{thm-main-1}, we further assume that $T\in C^{m+1,\alpha}(B_R(0))$, for some positive constant {$\alpha\in [0,1]$ with $m+\alpha>1$. Here, $C^{m, 0}$ denotes $C^m$.} Then, by a strategy similar to the proof of Theorem \ref{thm-main-2}, we will find that there exist functions {$c_1,c_2,\cdots,c_{m}\in C^{\infty}(\Sigma)\cap L^\infty(\Sigma)$} depending only on $n$, $g$, $V$, $m$, $\alpha$, and the derivatives of $T$ at $0$, such that 
\begin{equation*}\label{main-2-m-alpha}
\left|\frac{u(x)}{u_V(Tx)}-1-\sum^{m-[1-\alpha]}_{i=1} c_i\left(\frac{Tx}{|Tx|}\right)|Tx|^i\right|\leq
\begin{cases}
\begin{aligned}
 &Cd^{m+\alpha}_g(x, 0),\text{ if }\mu_1>m+\alpha,\\  
 &Cd^{m+\alpha}_g(x, 0)|\log{d_g(x, 0)}|,\text{ if }\mu_1=m+\alpha,\\
 &Cd^{\mu_1}_g(x, 0),\text{ if } m<\mu_1<m+\alpha,
\end{aligned}
\end{cases} 
\end{equation*}
{and if $\mu_1=m$ and $\alpha>0$, we have
$$\left|\frac{u(x)}{u_V(Tx)}-1-\sum^{m-1}_{i=1} c_i\left(\frac{Tx}{|Tx|}\right)|Tx|^i-c_{m}\left(\frac{Tx}{|Tx|}\right)|Tx|^{m}\log|Tx|\right|\leq Cd^{m}_g(x, 0).$$
}
\end{remark}

Note that when {$T=\id$ and the background manifold is the standard Euclidean space with $\Omega=V\cap B_2$}, Han, Jiang, and Shen \cite{HJS} {proved 
\begin{equation}\label{hjs-est}
    \left|\frac{u(x)}{u_V(x)}-1\right|\leq Cd_\Sigma^\tau|x|^{\mu_1},
\end{equation}
    for some constant $\tau>0$, where $d_{\Sigma}$ is the distance function to $\partial\Sigma$ on the sphere. They also obtained arbitrarily high-order asymptotic expansions for $u(x)/u_V(x)$. There are two key differences between our estimates and \eqref{hjs-est}:
    
    {{The first difference is that the expansion of $u/u_V$ in \eqref{hjs-est} does not have terms with order greater than 0 and less than $\mu_1$. 
    We point out that the discussion in \cite{HJS} relies heavily on the fact that both $u$ and $u_V$ satisfy the same equation.} The second difference is that the coefficient functions on ${\Sigma}$ in \eqref{hjs-est} and the {high-order} expansions of \cite{HJS} satisfy a decay of $O(d^{\tau}_{\Sigma})$,} which aligns with the Fredholm theory in the $L^2$-setting of the singular elliptic operator $L_1$ defined by \eqref{L1}, and it is essential for them to derive arbitrarily high-order asymptotic expansions. However, in our more general setting, $T$ is a perturbation of the identity map, and we can only guarantee that $c_i$ are bounded functions and cannot expect them to exhibit decay near the boundary, as demonstrated by {Example \ref{exampleA} and Example \ref{exampleB}}. Example \ref{exampleB} also implies that the term with order $1$ {does indeed appear}. Using the inner-outer gluing method, we prove a general existence result (Lemma \ref{solve-xi1}) for a class of uniformly degenerate elliptic equations in bounded Lipschitz domains. A key novelty of our existence result is that it allows the zero-order coefficient to be positive, or the non-homogeneous term to be non-$L^2$.}

This paper is organized as follows. Section \ref{Pre} presents some preliminary results. Section \ref{proof-thm-1} is devoted to the proof of Theorem \ref{thm-main-1}. In Section \ref{proof-thm-2}, we prove an existence result for uniformly degenerate elliptic equations in Lipschitz domains, which is then {applied} to prove Theorem \ref{thm-main-2}. {We construct some examples with no tangential decay in the left-hand side of \eqref{main-1} in Section \ref{nontanexample}, which guarantee that estimate \eqref{main-1} is optimal.} Finally, in the appendix (Section \ref{app}), we discuss some existence results and estimates for degenerate {or} singular elliptic equations that are crucial to the proof of Theorem \ref{thm-main-2}.
\medskip

{\noindent\it Acknowledgements: Weiming Shen is supported by NSFC Grant 12371208. Zhehui Wang is supported in part by NSFC Grant 12401247 and Guangdong Basic and Applied Basic Research Foundation (Grant No.\,{2023A1515110910}). Jiongduo Xie is partially supported by NSFC Grant W2531006, NSFC Grant 12250710674, NSFC Grant 12031012, and the Center of Modern Analysis-A Frontier Research Center of Shanghai. }

\section{Preliminaries}\label{Pre}
This section provides some preliminary results needed for the subsequent analysis. The following error estimates for the derivatives of a certain function under a $C^2$-diffeomorphism $T$ will be applied in the proof in the case $n=3$.
\begin{lemma}\label{estimate-derivatives-T}
Let $\Omega$, $V$, and $T$ be as in assumption \ref{assum}. Then, for any $f\in C^2(V\cap T(B_R(0)))$ and $x\in \Omega\cap B_R(0)$, we have
\begin{align}
    |(\partial_{x_i}(f\circ T))(x)-(\partial_{y_i}f)(Tx)| &\leq C_0|\nabla f(Tx)||x|,\label{condition-3}\\
    |(\partial^2_{x_ix_j}(f\circ T))(x)-(\partial^2_{y_iy_j}f)(Tx)|&\leq C_0(|\nabla f(Tx)|+|x||(\nabla^2f)(Tx)|),\label{condition-4}
\end{align}
where $C_0$ is a constant depending only on the $C^2$-norm of $T$.
\end{lemma}
\begin{proof}
Let $\varphi(x)=Tx-x$ and set $\varphi(x)=(\varphi^1(x), \cdots, \varphi^n(x))$. By assumptions, we have
$$\varphi(0)=0, \,\,\nabla\varphi(0)=0, \,\,\nabla^2\varphi=\nabla^2T.$$
Note 
\begin{equation}\label{vphi}|\partial_{x_i}\varphi^k(x)|=|\partial_{x_i}\varphi^k(x)-\partial_{x_i}\varphi^k(0)|\leq \|\nabla^2\varphi\|_\infty|x|= \|\nabla^2T\|_\infty|x|,\end{equation}
and 
\begin{align*}
    \partial_{x_i}(f\circ T(x))&=\partial_{x_i}(f(x+\varphi(x)))= \partial_{y_k}f(Tx)(\delta_{ki}+\partial_{x_i}\varphi^k(x))\\
    &=\partial_{y_i}f(Tx)+\partial_{y_k}f(Tx)\partial_{x_i}\varphi^k(x),
\end{align*}
we have
\begin{align*}| \partial_{x_i}(f\circ T(x))-\partial_{y_i}f(Tx)|&=|\partial_{y_k}f(Tx)\partial_{x_i}\varphi^k(x)|\leq |\nabla f(Tx)|\|\nabla^2T\|_\infty|x|.
\end{align*}
Furthermore, by \eqref{vphi}, $\nabla^2\varphi=\nabla^2T$, and
\begin{align*}
    &\partial_{x_ix_j}^2(f\circ T(x))=\partial_{x_i}[\partial_{y_j}f(Tx)+\partial_{y_k}f(Tx)\partial_{x_j}\varphi^k(x)]\\
    &=\partial_{y_ly_j}^2f(Tx)(\delta_{li}+\partial_{x_i}\varphi^l(x))+\partial_{y_ly_k}^2f(Tx)(\delta_{li}+\partial_{x_i}\varphi^l(x))\partial_{x_j}\varphi^k(x)\\&\qquad+\partial_{y_k}f(Tx)\partial_{x_ix_j}^2\varphi^k(x),
\end{align*}
we have
\begin{align*}
    &|\partial_{x_ix_j}^2(f\circ T(x))-\partial_{y_iy_j}^2f(Tx)|\\ &\leq |\partial_{y_ly_j}^2f(Tx)\partial_{x_i}\varphi^l(x)|+|\partial_{y_ly_k}^2f(Tx)(\delta_{li}+\partial_{x_i}\varphi^l(x))\partial_{x_j}\varphi^k(x)|\\&\qquad+|\partial_{y_k}f(Tx)\partial_{x_ix_j}^2\varphi^k(x)|\\
    &\leq  |\nabla^2 f(Tx)|\|\nabla^2T\|_\infty|x|+|\nabla^2 f(Tx)|(1+\|\nabla^2T\|_\infty|x|)\|\nabla^2T\|_\infty|x|\\&\qquad+|\nabla f(Tx)|\|\nabla^2 T\|_\infty\\
    &\leq  C_0(|\nabla f(Tx)|+|x||(\nabla^2f)(Tx)|),
\end{align*}
where $C_0$ is a constant depending only on $C^2$-norm of $T$.
\end{proof}
We also need the relationship between $|Tx|$ and $d_g(x, 0)$.
\begin{lemma}\label{distant}
Let $\Omega$, $V$, and $T$ be as in assumption \ref{assum}, we have
\begin{equation}\label{dg-|y|}
|Tx|=|x|+O(|x|^2)=d_g(x,0)+O(d^2_g(x,0)),    
\end{equation}
for $x$ near $0$.
\end{lemma}
\begin{proof}
    The first equality comes from the Taylor expansion of $T$ near $0$. 
    Note $(x_1, \cdots, x_n)$ is a normal coordinate system of $g$ with $0$ as the origin, 
    we know $\mathrm{exp}_0(tx):[0, 1]\to B_R(0)$ is a geodesic from $0$ to $x$ under the metric $g$, and $\mathrm{exp}_0(tx)=(tx_1, \cdots, tx_n)$ in the normal coordinate. We simply write $c(t)=\mathrm{exp}_0(tx)$. Then,
    \begin{align*}
        d_g(x, 0)=\int_0^1\sqrt{g(\dot{c}(t),\dot{c}(t))} \d t&=\int_0^1\sqrt{g_{ij}(tx)x_ix_j} \d t\\
        \text{(by Gauss lemma, $g_{ij}(tx)tx_itx_j=|tx|^2$)}&=\int_0^1|x|\d t=|x|.
    \end{align*}
\end{proof}
Let $\xi\in C^{\infty}(\Sigma)$ be a positive
function in $\Sigma$ satisfying \eqref{equation-Yamabe-2}-\eqref{boundary-condition-2}. By \cite[Lemma 2.4 and Lemma 2.5]{HJS}, the functions $\xi$ and $\rho=\xi^{-\frac{2}{n-2}}$ have the following property.
\begin{lemma}\label{rho-property}
Let $\Sigma\subsetneq\mathbb{S}^{n-1}$ be a Lipschitz domain, $d_\Sigma=d_{g_{\mathbb{S}^{n-1}}}(\cdot,\partial\Sigma)$, and $\xi\in C^{\infty}(\Sigma)$ be a positive
function in $\Sigma$ satisfying \eqref{equation-Yamabe-2}-\eqref{boundary-condition-2}. Then,
\begin{equation}\label{xi-estimate}
c_1\leq d^{\frac{n-2}{2}}_{\Sigma}\xi\leq c_2{\quad\mbox{ in } \Sigma,}
\end{equation}
and, for any integer $k\geq 0$,
\begin{equation}\label{xi-der-estimate}
d^{\frac{n-2}{2}+k}_{\Sigma}|\nabla^k_\theta \xi|\leq C{\quad\mbox{ in } \Sigma.}
\end{equation}
Moreover, set
\begin{equation}\label{rho-xi}
\rho=\xi^{-\frac{2}{n-2}},
\end{equation}
then, $\rho\in C^{\infty}(\Sigma)\cap \mathrm{Lip}(\bar\Sigma)$, $\rho>0$ in $\Sigma$, $\rho=0$ on $\partial\Sigma$, and 
\begin{equation}\label{rho-estimat}
c_3\leq \frac\rho{d_\Sigma}\leq c_4{\quad\mbox{ in } \Sigma.}
\end{equation}
Here, $c_1$, $c_2$, $c_3$, and $c_4$ are positive constants depending only on $n$ and $\Sigma$, and $C$ is a positive constant depending only on $n$, $k$, and $\Sigma$.
\end{lemma}

\section{Proof of Theorem \ref{thm-main-1}}\label{proof-thm-1}
\begin{proof}[Proof of Theorem \ref{thm-main-1}]
    For $x\in B_R(0)$, set $y=Tx$ and $v(y)=u(T^{-1}y)$. By a direct calculation, $v$ satisfies the following equation: $$\L v:=a_{ij}v_{ij}+b_iv_i+cv=\frac{n(n-2)}{4}v^{\frac{n+2}{n-2}} \quad \mbox{ in } V\cap T(B_R(0)),$$
    where 
\begin{align*}
a_{ij}=g^{kl}T^i_kT^j_l,\, b_i=g^{kl}T_{kl}^i+\frac{1}{\sqrt{\mathrm{det}g}}\partial_{x_k}(\sqrt{\mathrm{det}g}g^{kl})T^i_l,\, \mbox{ and } c=-\frac{n-2}{4(n-1)}S_g.
\end{align*}
Note 
$$g_{ij}=\delta_{ij}+O(|x|^2),\, \mbox{ and } g^{ij}=\delta_{ij}+O(|x|^2),$$ and $T\in C^2$ with $\nabla T(0)=\id$, then $a_{ij}\in C^1$, $b_i\in C^0$ and
\begin{equation}\label{coeff}
    a_{ij}=\delta_{ij}+O(|y|),\,\mbox{ and } b_i=O(1).
\end{equation}
Let 
\begin{equation*}
    \beta=\left\{\begin{aligned}
        1-\frac{2}{n-2}, \quad \mbox{ if } n\geq 4,\\
        0,\quad \mbox{ if } n=3,
    \end{aligned} \right.
\end{equation*}
and $r=|y|$. For some positive constants $A$ and $B$ to be determined, we set $$w(y)=u_V(y)+Au_V^\beta(y)+Bu_V(y)r, \quad  y\in V\cap T(B_R(0)).$$
Set 
$$D_4=(a_{ij}-\delta_{ij})w_{ij}+b_iw_i+cw,$$
by a direct calculation, we have
\begin{align*}&\L w-\frac{n(n-2)}{4}w^{\frac{n+2}{n-2}}=\Delta w-\frac{n(n-2)}{4}w^{\frac{n+2}{n-2}}+D_4\\&=\Delta w-\frac{n(n-2)}{4}u_V^{\frac{n+2}{n-2}}(y)\left(1+Au_V^{\beta-1}(y)+Br\right)^{\frac{n+2}{n-2}}+D_4 \\
&\leq\Delta w-\frac{n(n-2)}{4}u_V^{\frac{n+2}{n-2}}(y)\left(1+\frac{n+2}{n-2}(Au_V^{\beta-1}(y)+Br)\right)+D_4\\
&= \Delta\left(Au_V^\beta(y)+Bu_V(y)r\right)-\frac{n(n+2)}{4}u_V^{\frac{n+2}{n-2}}(y)\left(Au_V^{\beta-1}(y)+Br\right)+D_4\\&:= D_1+D_2+D_3+D_4,\end{align*}
where 
\begin{align*}
    D_1&=-\frac{n(n+2)}{4}\left(1-\frac{n-2}{n+2}\beta\right)Au_V^{\beta+\frac{4}{n-2}}+A\beta(\beta-1)u_V^{\beta-2}|\nabla u_V|^2,\\
    D_2&=-\frac{n(n+2)}{4}\left(1-\frac{n-2}{n+2}\right)Bru_V^{\frac{n+2}{n-2}},\\
    D_3&=2B\nabla r\cdot\nabla u_V+B\frac{n-1}{r}u_V.
\end{align*}
By \eqref{coeff}, we know
\begin{align*}
|D_4|&=O(r)|\nabla^2 w|+O(1)|\nabla w|+ O(1)|w|.
\end{align*}
According to \cite[Remark 3.3, Lemma 3.4]{HS}, we know that there is a constant $C_1>1$ depending only on $n$ and the size of the exterior cones such that for $y\in V\cap T(B_R(0))$,
\begin{equation}\label{1-2-order}
   d^2(y)|\nabla^2u_V(y)|+d(y)|\nabla u_V(y)|\leq C_1 u_V(y),
\end{equation}
and 
\begin{equation}\label{0-order}
C_1^{-1}\leq d^{\frac{n-2}{2}}(y)u_V(y)\leq 2^{\frac{n-2}{2}},
\end{equation}
where $d(y)=d(y, \partial V)$. Note $d\leq r\leq1$, we have
\begin{align*}
|w|&\leq \frac{u_V}{d}\left(r+A u_V^{\beta-1}r+Br^2\right),\\
|\nabla w|&\leq |\nabla u_V|+A\beta u_V^{\beta-1}|\nabla u_V|+Br|\nabla u_V|+Bu_V\\
&\leq C_1\frac{u_V}{d}\left(1+A\beta u_V^{\beta-1}+2Br\right),
\end{align*}
and
\begin{align*}
|\nabla^2 w|&\leq (1+A\beta u_V^{\beta-1}+Br)|\nabla^2 u_V|+A\beta|\beta -1| u_V^{\beta-2}|\nabla u_V|^2+2B|\nabla u_V|+2B\frac{u_V}{r}\\
&\leq C_1\frac{u_V}{d^2}\left(1+A\beta u_V^{\beta-1}+C_1A\beta|\beta -1| u_V^{\beta-1}+B(3+\frac{2}{C_1})r\right).
\end{align*}
Hence,
\begin{equation}
    |D_4|\leq C_2\frac{u_Vr}{d^2}(1+A(\beta+1) u_V^{\beta-1}+Br),
\end{equation}
where $C_2>0$ is a constant depending only on $n$ and the size of the exterior cones.
Note $0\leq\beta<1$, we have
\begin{equation*}
D_1\leq -\frac{n(n+2)}{4}\left(1-\frac{n-2}{n+2}\beta\right)Au_V^{\beta+\frac{4}{n-2}}
    =\left\{
        \begin{aligned}
    -\frac{n(n+2)}{4}\left(1-\frac{n-2}{n+2}\beta\right)Au_Vu_V^{\frac{2}{n-2}},&\quad n\geq 4, \\
    -\frac{n(n+2)}{4}\left(1-\frac{n-2}{n+2}\beta\right)Au_Vu_V^{3},&\quad n=3.
\end{aligned}\right.
\end{equation*}
By choosing $0<r_1<R$ small enough, such that $u_V(y)\geq 1$ for any $y\in V\cap T(B_R(0))\cap B_{r_1}(0)$, we have
\begin{equation}
    D_1\leq -\frac{n(n+2)}{4}\left(1-\frac{n-2}{n+2}\beta\right)\frac{C_1^{-\frac{2}{n-2}}Au_V}{d}, \quad \mbox{for } y\in V\cap T(B_R(0))\cap B_{r_1}(0).
\end{equation} 
Similarly, we have
\begin{equation}
    -4n \frac{Bru_V}{d^2}\leq  D_2\leq -C_1^{-\frac{4}{n-2}}n\frac{Bru_V}{d^2},
\end{equation}
and
\begin{equation}
  D_3\leq (2C_1+n-1)\frac{Bu_V}{d}.
\end{equation}
Let $$A=\frac{4}{n(n+2)}(1-\frac{n-2}{n+2}\beta)^{-1}\left(2C_1+n-1\right)C_1^{\frac{4}{n-2}}B:=C_3B,$$
we get
\begin{equation}\label{D1+3}
    D_1+D_3\leq 0.
\end{equation}
Note
\begin{align*}
    D_2+D_4\leq C_2\frac{u_Vr}{d^2}\left(1+A(\beta+1) u_V^{\beta-1}+Br-BC_1^{-\frac{4}{n-2}}C_2^{-1}n\right),
\end{align*}
and 
\begin{equation*}
A(\beta+1) u_V^{\beta-1}\leq\left\{
        \begin{aligned}
    2AC_1^{\frac{2}{n-2}}d,&\quad n\geq 4, \\
    AC_1d^{\frac{n-2}{2}},&\quad n=3,
\end{aligned}\right.
\end{equation*}
we have
\begin{align*}
    D_2+D_4&\leq \left\{
\begin{aligned}
    C_2\frac{u_Vr}{d^2}\left(1+2AC_1^{\frac{2}{n-2}}r+Br-BC_1^{-\frac{4}{n-2}}C_2^{-1}n\right),&\quad n\geq 4, \\
    C_2\frac{u_Vr}{d^2}\left(1+AC_1d^{\frac{n-2}{2}}+Br-BC_1^{-\frac{4}{n-2}}C_2^{-1}n\right),&\quad n=3,
\end{aligned}\right.\\ &\leq C_2\frac{u_Vr}{d^2}\left(1+\left((2C_3C_1^{\frac{2}{n-2}}+1)r+C_3C_1r^{\frac{1}{2}}-C_1^{-\frac{4}{n-2}}C_2^{-1}n\right)B\right).
\end{align*}
Choose $0<r_2<r_1$ small enough and $B>>1$ such that 
$$(2C_3C_1^{\frac{2}{n-2}}+1)r_2+C_3C_1r_2^{\frac{1}{2}}-C_1^{-\frac{4}{n-2}}C_2^{-1}n \leq -\frac{C_1^{-\frac{4}{n-2}}C_2^{-1}n}{2},\quad B\geq \frac{2C_1^{\frac{4}{n-2}}C_2}{n},$$
we get \begin{equation}D_2+D_4\leq 0, \quad \mbox{ for } y\in V\cap T(B_R(0))\cap B_{r_2}(0),\end{equation}
which together with \eqref{D1+3} imply 
$$\L w\leq \frac{n(n-2)}{4}w^{\frac{n+2}{n-2}} \quad \mbox{ in } V\cap T(B_R(0))\cap B_{r_2}(0).$$

For any $s>0$ and $y\in B_s(0)$, we set $u_s(y)=(\frac{2s}{s^2-|y|^2})^{\frac{n-2}{2}}$. By direct calculations, we know that
\begin{align*}
    \Delta u_s&=\frac{n(n-2)}{4}u_s^{\frac{n+2}{n-2}} \quad \mbox{ in } B_s(0),\\
    \L (2u_s)&\leq\frac{n(n-2)}{4}(2u_s)^{\frac{n+2}{n-2}} \quad \mbox{ in } B_s(0),
\end{align*}
for $s$ small enough.
Therefore, by choosing $0<s<r_2$ small enough, we have that in $V\cap T(B_R(0))\cap B_{s}(0)$,
\begin{align}\label{thm1com}
    \L (w+2u_s)\leq\frac{n(n-2)}{4}(w^{\frac{n+2}{n-2}}+(2u_s)^{\frac{n+2}{n-2}})\leq \frac{n(n-2)}{4}(w+2u_s)^{\frac{n+2}{n-2}}.
\end{align}

Without loss of generality, we assume $$V\cap T(B_R(0))\cap B_{s}(0)=\{(y', y_n)\in B_s(0):y_n>f(y')\},$$ where $f$ is Lipschitz and $f(0')=0$. Set $F(\cdot)=\frac{n(n-2)}{4}(\cdot)^{\frac{n+2}{n-2}}$ and $$v^\epsilon(y)=v(y_1,\cdots,y_{n-1},y_n+\epsilon).$$
By \eqref{thm1com} and
$\L v=F(v)$, we have for any $y\in V\cap T(B_R(0))\cap B_{s}(0),$
\begin{align*}
    a_{ij}(w+2u_s-v^\epsilon)_{ij}+b_i(w+2u_s-v^\epsilon)_i&\leq F(w+2u_s)-F(v^\epsilon)-c(w+2u_s-v^\epsilon)\\
    &\leq F'(\varphi)(w+2u_s-v^\epsilon)-c(w+2u_s-v^\epsilon),
\end{align*}
where $\varphi$ is between $w+2u_s$ and $v^\epsilon$. Recall $c=-\frac{n-2}{4(n-1)}S_g$, we may assume $|c|\leq C_g$, where $C_g>0$ is a constant depending only on $n$ and $g$. In addition, we can take $s>0$ small enough such that
\begin{equation}\label{compare-G}
    \frac{n(n+2)}{4}(w+2u_s)^{\frac{4}{n-2}}>C_g\quad \mbox{ in } V\cap T(B_R(0))\cap B_{s}(0).
\end{equation}
Set $G=V\cap T(B_R(0))\cap B_{s}(0)\cap\{y: v^\epsilon(y)>(\frac{4C_g}{n(n+2)})^{\frac{n-2}{4}}\}$, by \eqref{compare-G}, we have
$$F'(\varphi)>C_g\quad\mbox{ in } G,$$
which implies $F'(\varphi)-c>0$ in $G$.
Note $u_s=+\infty$ on $\partial B_s$ and $w=+\infty$ on {$\partial V\cap T(B_R(0))$. Since $0<s<<R$}, it follows from \eqref{compare-G} that $w+2u_s>v^\epsilon$ on $\partial G$. Then, by
$$a_{ij}(w+2u_s-v^\epsilon)_{ij}+b_i(w+2u_s-v^\epsilon)_i-(F'(\varphi)-c)(w+2u_s-v^\epsilon)\leq 0\quad \mbox{ in } G,$$
and the maximum principle, we know
$$w+2u_s-v^\epsilon\geq 0\quad \mbox{ in } G.$$
On the other hand, by \eqref{compare-G}, in $V\cap T(B_R(0))\cap B_{s}(0)\cap\{y: v^\epsilon(y)\leq(\frac{4C_g}{n(n+2)})^{\frac{n-2}{4}}\}$,  we have 
$$v^\epsilon\leq\left(\frac{4C_g}{n(n+2)}\right)^{\frac{n-2}{4}}\leq w+2u_s.$$
So, $$w+2u_s-v^\epsilon\geq 0\quad \mbox{ in } V\cap T(B_R(0))\cap B_{s}(0),$$ and further, by letting $\epsilon\to 0$, we get
\begin{equation}\label{thm-rhs}
    v \leq w+2u_s \quad \mbox{ in } V\cap T(B_R(0))\cap B_{s}(0).
\end{equation}
 Note that in $B_{\frac{s}{2}}(0)$, $u_s\leq C(n, s)$, where $C(n, s)>0$ is a constant depending only on $n$ and $s$, then for any $0<s'\leq \frac{s}{2},$ by \eqref{thm-rhs} and \eqref{0-order} we have
\begin{align*}
    v-u_V&\leq \left\{
\begin{aligned}
    u_V\left(AC_1^{\frac{2}{n-2}}r+Br+2C(n, s)C_1r^{\frac{n-2}{2}}\right),&\quad n\geq 4, \\
    u_V\left(Br+(A+2C(n, s))C_1r^{\frac{1}{2}}\right),&\quad n=3,
\end{aligned}\right.\\
&\leq\left\{
\begin{aligned}
   C u_V|y|,&\quad n\geq 4, \\
    C u_V|y|^{\frac{1}{2}},&\quad n=3,
\end{aligned}\right.
\end{align*}
where $C>0$ is a constant depending only on $n$, the size of the exterior cones, and $g$. By a process similar to that in the proof of \cite[Theorem 6.2]{SW} and note that we have \eqref{condition-3}-\eqref{condition-4} by Lemma \ref{estimate-derivatives-T}, we can improve the estimate in dimension $3$ and get 
$$v-u_V\leq Cu_V|y| \quad \mbox{ in } V\cap T(B_R(0))\cap B_{s'}(0).$$
By taking $\tilde{w}=u_V-Au_V^\beta-Bu_Vr$ and repeating a similar argument, we can prove for $s''>0$ small enough,
$$v-u_V\geq -Cu_V|y|\quad \mbox{ in } V\cap T(B_R(0))\cap B_{s''}(0),$$
where $C>0$ is a constant depending only on $n$, the size of the exterior cones, and $g$.
Therefore, by Lemma \ref{distant}, $$\left|\frac{u(x)}{u_V(Tx)}-1\right|\leq C|Tx|\leq Cd_g(x, 0) \quad \mbox{ in } \Omega\cap B_{r_0}(0),$$ where $r_0$ is a constant small enough.

\end{proof}


\section{Proof of Theorem \ref{thm-main-2}}\label{proof-thm-2}
In this section, we study the first order expansion of local positive solutions of \eqref{equation-Yamabe-1}-\eqref{boundary-condition-1} near singular points on $\partial\Omega$ {when $T$ has higher regularity.} Formal calculations lead to an elliptic operator such as \eqref{L0} or \eqref{L1}. However, we cannot use Theorem \ref{outer-problem-theorem} or Lemma \ref{inner-problem-lemma} directly since $c$ in \eqref{L0} may be positive or $f$ in \eqref{L1-eigenvalue} may not belong to $L^2(\Sigma)$.
Specifically, for some constant $\sigma\in\mathbb R$ and positive integer $k$, we study the following elliptic equation on $\Sigma$:
\begin{equation}\label{xi1-equation}
\rho^2\Delta_\theta h-\frac{n(n+2)}{4}h+\sigma_k\rho^2h=\rho^2 (F+\sigma\phi_1),
\end{equation}
where $\rho$ is defined in \eqref{rho-xi},
\begin{equation}\label{sigma-k}
\sigma_k:= \Big(k-\frac n 2+1\Big)\Big(k+\frac n2-1\Big),   
\end{equation}
$\phi_1\in H^1_0(\Sigma)\cap C^{\infty}(\Sigma)\cap C(\bar\Sigma)$ is the eigenfunction corresponding to the first eigenvalue $\lambda_1$ of $L_1$ in \eqref{L1}, satisfying $\phi_1>0$ on $\Sigma$ with $\|\phi_1\|_{L^2(\Sigma)}=1$, and $F\in C^{\infty}(\Sigma)$ satisfies 
\begin{equation}\label{estimate-F}
|F|\leq \bar{C} \xi^{\frac{n+2}{n-2}}\quad\text{on }\Sigma,
\end{equation}
for some positive constant $\bar C>0$.
In the following, we split equation \eqref{xi1-equation} into two parts: the first is solved via Theorem \ref{outer-problem-theorem}, and the second via Lemma \ref{inner-problem-lemma}.
\begin{lemma}\label{solve-xi1}
Let $\Sigma\subsetneq\mathbb{S}^{n-1}$ be a Lipschitz domain. Suppose $\xi\in C^{\infty}(\Sigma)$ is a positive
function on $\Sigma$ satisfying \eqref{equation-Yamabe-2}-\eqref{boundary-condition-2}, and $F\in C^{\infty}(\Sigma)$ satisfies \eqref{estimate-F}. For a given positive integer $k$, denote by $\sigma_k$ the constant defined in
\eqref{sigma-k} with $\sigma_k\leq \lambda_1$.

$\mathrm{(i)}$ If $\sigma_k<\lambda_1$, then for $\sigma=0$, there exists a solution $h$ of \eqref{xi1-equation};

$\mathrm{(ii)}$ If $\sigma_k=\lambda_1$, then for some $\sigma\in\mathbb R$, there exists a solution $h$ of \eqref{xi1-equation}.

In both cases, \begin{equation}\label{estimate-xi1}
|h|\leq C\xi\quad\text{on }\Sigma,
\end{equation}
where $C$ is a positive constant depending only on $n$, $k$, $\Sigma$, and $\bar C$.
{Moreover, if we further assume that for some constant $\beta\in(0,1)$,
\begin{equation}\label{F-Holder}
d^\beta_{\Sigma}(x)[\rho^2F]_{C^\beta(B_{d_\Sigma(x)/4}(x))}\leq \tilde{C}\|\xi\|_{L^\infty(B_{d_{\Sigma}(x)/2}(x))},
\end{equation}
for any $x\in\Sigma$, then
\begin{equation}\label{estimate-xi1-derivative}
d_{\Sigma}|\nabla_\theta h|+d^2_{\Sigma}|\nabla^2_\theta h|+d_\Sigma^{2+\beta}[\nabla^2_\theta h]_{C^{\beta}(B_{d_\Sigma(\cdot)/2}(\cdot))}\leq C\xi\quad\text{on }\Sigma,
\end{equation}
where $C$ is a positive constant depending only on $n$, $k$, $\beta$, $\Sigma$, $\bar C$, and $\tilde C$. 
}
\end{lemma}
\begin{proof}
Set $\rho_0=\sqrt{n/(2k^2)}$. Note that on $\{\rho<\rho_0\}$,
\begin{equation}\label{c-condition}
\begin{aligned}
&-\frac{n(n+2)}{4}+\Big(k-\frac n 2+1\Big)\Big(k+\frac n2-1\Big)\rho^2+\frac{(n-2)^2}{4}\rho^2+\frac{n(n-1)}{4}\\
&\qquad=-\frac{3n}4 +k^2\rho^2<-\frac{3n}4 +\frac n2=-\frac{n}4<0,
\end{aligned}
\end{equation}
we can take $c\in \mathrm{Lip}(\bar\Sigma)\cap C^\infty(\Sigma)$ such that
\begin{equation}\label{c-def}
\begin{cases}
\begin{aligned}
c&=-\frac{n(n+2)}{4}+\sigma_k\rho^2\quad\text{on }\{\rho<\rho_0\},\\
c&\leq -\frac{(n-2)^2}{4}\rho^2-\frac{n(n-1)}{4}\quad\text{on }\Sigma.
\end{aligned}
\end{cases}
\end{equation}

Step 1. Set $L_0=\rho^2\Delta_\theta+c$, we firstly solve 
\begin{equation}\label{tilde-xi-eq}
L_0\tilde{h}=\rho^2 F.
\end{equation}
By Lemma \ref{rho-property}, \eqref{equation-Yamabe-2}, and \eqref{c-def}, we have
$$
\rho\nabla_\theta\log \xi=\rho\frac{\nabla_\theta\xi}\xi=-\frac{n-2}2\nabla_\theta\rho\in L^\infty(\Sigma),
$$
and
\begin{align*}
L_0\xi&=\rho^2\Delta_\theta\xi+c\xi\\
&=\rho^2\Big(\frac{(n-2)^2}{4}\xi+\frac{n(n-2)}{4}\xi^{\frac{n+2}{n-2}}\Big)+c\xi\\
&=\xi\Big(\frac{(n-2)^2}{4}\rho^2+\frac{n(n-2)}{4}+c\Big)\leq -\frac n 4\xi.
\end{align*}
Now, we can take $\psi=\xi$ and $\delta=n/4$ in Theorem \ref{outer-problem-theorem}. Hence, we obtain an $\tilde{h}\in C^{\infty}(\Sigma)$ satisfying \eqref{tilde-xi-eq} with
\begin{equation}\label{tilde-xi-1-estimate}
\Big\|\frac{\tilde{h}}{\xi}\Big\|_{L^{\infty}(\Sigma)}\leq\frac1\delta \Big\|\frac{\rho^2 F}{\xi}\Big\|_{L^{\infty}(\Sigma)}\leq C,
\end{equation}
where we use \eqref{estimate-F}.

Step 2. Set $L_1=\Delta_\theta-\frac{n(n+2)}{4\rho^2}$, we secondly solve 
\begin{equation}\label{bar-xi-eq}
L_1\bar{h}+\sigma_k\bar{h}=G,
\end{equation}
where  
$$
G=\sigma\phi_1+\tilde{G},
$$
with 
$$
\tilde{G}=\tilde{h}\Big(\frac{c}{\rho^2}+\frac{n(n+2)}{4\rho^2}-\sigma_k\Big).
$$
By \eqref{c-def}, we have $\tilde{G}=0$ on $\{\rho<\rho_0\}$ and $\tilde{G}\in C^{\infty}(\bar\Sigma)$.

(i) If $\sigma_k< \lambda_1$, we take $\sigma=0$. 

(ii) If $\sigma_k=\lambda_1$, we take $\sigma=-(\tilde{G},\phi_1)_{L^2(\Sigma)}$, then $(G,\phi_1)_{L^2(\Sigma)}=0$. 

Hence, in both cases, we can apply Lemma \ref{inner-problem-lemma}(2) with $\lambda=\sigma_k$ and $f=G$ to find a $\bar h\in C^{\infty}(\Sigma)$ satisfying \eqref{bar-xi-eq} with
\begin{align*}
\|\bar h\|_{H^1_0(\Sigma)}\leq C\|G\|_{L^{2}(\Sigma)}\leq C\|\tilde{G}\|_{L^{2}(\Sigma)}\leq C.
\end{align*}
In addition, by Lemma \ref{inner-problem-lemma}(4), we have
\begin{equation}\label{bar-xi-1-estimate}
\|\bar h\|_{L^\infty(\Sigma)}\leq C(\|\bar h\|_{L^{2}(\Sigma)}+1)\leq C.
\end{equation}

Step 3. Take $h=\tilde h+\bar h$. By \eqref{tilde-xi-eq} and \eqref{bar-xi-eq}, we can verify that $h\in C^{\infty}(\Sigma)$ satisfies \eqref{xi1-equation}. In addition, by combining \eqref{tilde-xi-1-estimate} and \eqref{bar-xi-1-estimate}, we obtain \eqref{estimate-xi1}.
Next, {we assume that $F$ satisfies \eqref{F-Holder}.} 
Recall $\xi=\rho^{-(n-2)/2}\sim d_{\Sigma}^{-(n-2)/2}$, and we note for any $z\in B_{d_{\Sigma}(x)/2}(x)$,
$$
\xi(z)\sim d^{-\frac{n-2}{2}}_{\Sigma}(z)\sim d^{-\frac{n-2}{2}}_{\Sigma}(x)\sim\xi(x).
$$
By applying the scaled $C^{2,\alpha}$-estimate to \eqref{xi1-equation} and using \eqref{estimate-F}-\eqref{F-Holder} and \eqref{eigenvalue-function-estimate}, we get
\begin{align*}
&{d_{\Sigma}(x)\|\nabla_\theta h\|_{L^\infty(B_{d_{\Sigma}(x)/8}(x))}+d^2_{\Sigma}(x)\|\nabla^2_\theta h\|_{L^\infty(B_{d_{\Sigma}(x)/8}(x))}+d_\Sigma^{2+\beta}(x)[\nabla^2_\theta h]_{C^{\beta}(B_{d_\Sigma(x)/8}(x))}}\\
&\qquad\leq C(\|h\|_{L^\infty(B_{d_{\Sigma}(x)/4}(x))}+\frac{d^2_{\Sigma}(x)}{d^2_{\Sigma}(x)}\|\rho^2(F+\sigma\phi_1)\|_{L^\infty(B_{d_{\Sigma}(x)/4}(x))}\\
&\qquad\qquad+\frac{d^{2+\beta}_{\Sigma}(x)}{d^{2}_{\Sigma}(x)}[\rho^2(F+\sigma\phi_1)]_{C^\beta (B_{d_{\Sigma}(x)/4}(x))})\\
&\qquad\leq C\|\xi\|_{L^\infty(B_{d_{\Sigma}(x)/2}(x))}.
\end{align*}
Therefore, we obtain \eqref{estimate-xi1-derivative}.
\end{proof}
{
\begin{remark}\label{premu1}
In particular, when $k=1$, we have $\sigma_1=-n(n-4)/4$. Note that $-n(n-4)/4\leq0<\lambda_1$ for $n\geq 4$ and $-n(n-4)/4=3/4<\lambda_1$ for $n=3$ by Lemma \ref{inner-problem-lemma}(1) and Lemma \ref{inner-problem-lemma}(3). 

We further observe that when $\sigma_k = \lambda_1$ for some integer $k>1$, the solution of equation \eqref{xi1-equation} is not unique. In fact, if $h$ is a solution of \eqref{xi1-equation}, then $h + t\phi_1$ is also a solution for any $t \in \mathbb{R}$. 
\end{remark}}
Now, we begin to prove Theorem \ref{thm-main-2}.
\begin{proof}[Proof of Theorem \ref{thm-main-2}]
We adopt the notation from the proof of Theorem \ref{thm-main-1}. Recall that $v(y)=u(T^{-1}y)$. Set $r=|y|=|Tx|$. By Theorem \ref{thm-main-1}, there exist some constants $r_0>0$ and $B>0$ such that 
\begin{equation}\label{estimate-v-u_v}
|v(y)-u_V(y)|\leq Bru_V(y),\quad\text{for any }r\leq r_0.     
\end{equation}
Let $\lambda_1$ be the first eigenvalue of $L_1$ in \eqref{L1} and let $\phi_1\in H^1_0(\Sigma)\cap C^{\infty}(\Sigma)\cap C(\bar\Sigma)$ be the corresponding eigenfunction satisfying $\phi_1>0$ in $\Sigma$ with $\|\phi_1\|_{L^2(\Sigma)}=1$.  Note that the existence of such a function $\phi_1$ is guaranteed by Lemma \ref{inner-problem-lemma}, {and $|\phi_1|\leq C'\xi$ for some constant $C'>0$ depending only on $n$, $\lambda_1$, and $\Sigma$.}

\textbf{Case 1. $\mu_1>2.$ }

Let $A_0,A_1>0$ be constants and let $\xi_1\in C^\infty(\Sigma)$ be a function to be determined satisfying estimates of the form \eqref{estimate-xi1} and \eqref{estimate-xi1-derivative}, i.e., {for some $\beta\in(0,1)$,
\begin{equation}\label{scalexi}
|\xi_1|+d_\Sigma|\nabla_\theta\xi_1|+d_\Sigma^2|\nabla^2_\theta\xi_1|+d_\Sigma^{2+\beta}[\nabla^2_\theta \xi_1]_{C^{\beta}(B_{d_\Sigma(\cdot)/2}(\cdot))}\leq C_0\xi,
\end{equation}
for some constants $C_0>0$ depends only on $n$, $\beta$, $g$, $V$, and the derivatives of $T$ at $0$.
We set   
$$
w=u_V+\xi_1 r^{-\frac n2+2}+A_0u_V r^2+A_1\phi_1 r^{-\frac n 2+3}=u_{V}(1+R_1), 
$$
where
$$
R_1:=\frac{\xi_1}{\xi}r+A_0r^2+A_1\frac{\phi_1}{\xi}r^2.
$$
We will choose $A_0$, $A_1$, $\xi_1$, and $\bar{r}_1>0$ small enough so that
\begin{equation}\label{R-estimate}
C_0\bar r_1+A_0\bar r^2_1+A_1C'\bar r^2_1\leq \frac{1}{2},
\end{equation}
hence we have, for any $r\leq \bar r_1$,
\begin{equation}\label{R1-exp}
\left|(1+R_1)^{\frac{n+2}{n-2}}-1-\frac{n+2}{n-2}R_1\right|\leq C(n)R^2_1,
\end{equation}
where $C(n)>0$ depends only on $n$. 

We suppose \eqref{R-estimate} holds, then by \eqref{equation-Yamabe-2} and $L_1\phi_1+\lambda_1\phi_1=0$, a direct computation yields
\begin{equation}\label{LN-super}
\begin{aligned}
&\Delta w-\frac{1}{4}n(n-2)w^{\frac{n+2}{n-2}}\\
&=\Delta w-\frac{1}{4}n(n-2)u_V^{\frac{n+2}{n-2}}(1+R_1)^{\frac{n+2}{n-2}}\\
&\leq \Delta w-\frac{1}{4}n(n-2)u_V^{\frac{n+2}{n-2}}\Big[1+\frac{n+2}{n-2}R_1-C(n)R^2_1\Big]\\
&=\Big(\Delta_\theta \xi_1-\frac{n(n+2)}{4\rho^2}\xi_1+\frac{n(4-n)}{4}\xi_1\Big)r^{-\frac{n}{2}}+A_0(4\xi-n\xi^{\frac{n+2}{n-2}})r^{-\frac{n}{2}+1}\\
&\qquad+A_1\Big(\Big(-\frac{n}{2}+3\Big)\Big(\frac{n}{2}+1\Big)-\lambda_1\Big)\phi_1 r^{-\frac{n}{2}+1}+\frac{1}{4}n(n-2)C(n)u_V^{\frac{n+2}{n-2}}R^2_1.
\end{aligned}
\end{equation}
By $T\in C^3$,} we have $a_{ij}\in C^2,b_i\in C^1$, and 
$$a_{ij}=\delta_{ij}+a_{ij,k}y_k+O(r^2),\mbox{ and }\ b_i=b_{i,0}+O(r).$$
Hence, by \eqref{xi-estimate}, \eqref{xi-der-estimate}, \eqref{scalexi}, and \eqref{eigenvalue-function-estimate}, we get
\begin{equation}\label{Yamabe-super}
\begin{aligned}
&\L w-\frac{n(n-2)}{4}w^{\frac{n+2}{n-2}}\\
&=(\delta_{ij}+a_{ij,k}y_k+O(r^2)) w_{ij}+(b_{i,0}+O(r))w_i+cw-\frac{n(n-2)}{4}w^{\frac{n+2}{n-2}}\\
&=\Big(\Delta w-\frac{n(n-2)}{4}w^{\frac{n+2}{n-2}}\Big)+a_{ij,k}y_k w_{ij}+O(r^2)w_{ij}+b_{i,0}w_i+O(r)w_i+cw\\
&=\Big(\Delta w-\frac{n(n-2)}{4}w^{\frac{n+2}{n-2}}\Big)+a_{ij,k}y_k (u_{V})_{ij}+b_{i,0}(u_{V})_i+\xi^{\frac{n+2}{n-2}}O(r^{-\frac n2+1})\\
&\qquad+(A_0+A_1+1)\xi^{\frac{n+2}{n-2}}O(r^{-\frac n2+2}).
\end{aligned}
\end{equation}
Note that $F_1:=-r^{\frac n 2}(a_{ij,k}y_k (u_{V})_{ij}+b_{i,0}(u_{V})_i)\in C^\infty(\Sigma)$ is a linear combination of $\nabla_\theta \xi$ and $\nabla^2_\theta \xi$. It follows from \eqref{xi-der-estimate} that $F_1$ satisfies \eqref{estimate-F} and \eqref{F-Holder}. {By Remark \ref{premu1}, we can apply Lemma \ref{solve-xi1}(i) with $F=F_1$ and $k=1$ to obtain the desired $\xi_1$ satisfying \eqref{xi1-equation} with $\sigma=0$, and \eqref{scalexi}.} Therefore, {by substituting \eqref{LN-super} into \eqref{Yamabe-super}, we have}
\begin{align*}
&\L w-\frac{n(n-2)}{4}w^{\frac{n+2}{n-2}}\\
&\leq\Big(\Delta_\theta \xi_1-\frac{n(n+2)}{4\rho^2}\xi_1+\frac{n(4-n)}{4}\xi_1-F_1\Big)r^{-\frac{n}{2}}\\
&\qquad+\Big\{\xi^{\frac{n+2}{n-2}}O(1)+A_0(4\xi-n\xi^{\frac{n+2}{n-2}})+A_1\Big(\Big(-\frac{n}{2}+3\Big)\Big(\frac{n}{2}+1\Big)-\lambda_1\Big)\phi_1 \Big\}r^{-\frac n2+1}\\
&\qquad+(A_0+A_1+1)\xi^{\frac{n+2}{n-2}}O(r^{-\frac n2+2}).
\end{align*}
{where the first term on the right hand side vanishes.}

Step 1. Set $A_1=kA_0$. First, we have $|O(1)|\leq C_1$ for $r\leq r_0$. Then, we choose $A_0\geq\max\{2C_1/n,1\}$. We claim that by choosing $k>0$ large enough, we have
\begin{equation}\label{choose-1}
\xi^{\frac{n+2}{n-2}}O(1)+A_0(4\xi-n\xi^{\frac{n+2}{n-2}})+A_1\Big(\Big(-\frac{n}{2}+3\Big)\Big(\frac{n}{2}+1\Big)-\lambda_1\Big)\phi_1 \leq -\frac n4A_0\xi^{\frac{n+2}{n-2}}.
\end{equation}

By \eqref{boundary-condition-2}, we can take $\delta>0$ small enough such that $\xi^{\frac{4}{n-2}}>16/n$ on $\Sigma_\delta:=\{x\in\Sigma:d_{\Sigma}(x)<\delta\}$. Hence, we have 
$$4\xi-\frac n2\xi^{\frac{n+2}{n-2}}<-\frac n4\xi^{\frac{n+2}{n-2}}\quad\text{on }\Sigma_\delta.$$ Recall \eqref{mu1-def}, and note that $\mu_1>2$ implies $(-n/2+3)(n/2+1)-\lambda_1<0$. Hence, on $\Sigma_\delta$ the left-hand side of \eqref{choose-1} is controlled by
$$
C_1\xi^{\frac{n+2}{n-2}}+A_0(4\xi-n\xi^{\frac{n+2}{n-2}})\leq A_0\Big(\frac n2\xi^{\frac{n+2}{n-2}}+4\xi-n\xi^{\frac{n+2}{n-2}}\Big)\leq-\frac n4A_0\xi^{\frac{n+2}{n-2}}.
$$
On the other hand, we have
$$
\phi_1\geq c_\delta,\quad \xi\leq C_\delta\quad\text{on }\Sigma\backslash\Sigma_\delta,
$$
for some positive constants $c_\delta$ and $C_\delta$. Hence, on $\Sigma\backslash\Sigma_\delta$ the left-hand side of \eqref{choose-1} is controlled by
$$
\frac n2A_0C^{\frac{n+2}{n-2}}_\delta+4A_0C_\delta-nA_0\xi^{\frac{n+2}{n-2}}+A_0k\Big(\Big(-\frac{n}{2}+3\Big)\Big(\frac{n}{2}+1\Big)-\lambda_1\Big)c_\delta\leq  -nA_0\xi^{\frac{n+2}{n-2}},
$$
where we take 
\begin{equation}\label{k-large}
k=\Big(\frac n2C^{\frac{n+2}{n-2}}_\delta+4C_\delta\Big)\Big/\Big(\Big(\frac{n}{2}-3\Big)\Big(\frac{n}{2}+1\Big)+\lambda_1\Big)c_\delta.
\end{equation}
In summary, by choosing $A_0\geq \max\{2C_1/n,1\}$ and $k$ as \eqref{k-large}, we obtain \eqref{choose-1}.

Step 2. There exists a constant $r'_0>0$ small enough such that 
\begin{equation}\label{choose-2}
\begin{aligned}
(A_0+A_1+1)O(r^{-\frac n2+2})&\leq (k+2)A_0r |O(r^{-\frac n2+1})|\\
&\leq \frac n8A_0r^{-\frac n2+1}\quad\text{for any }r\leq r'_0.\end{aligned}
\end{equation}

Step 3. By \eqref{scalexi}, we have 
\begin{equation}\label{boundary-infty}
|\xi_1 r^{-\frac n2+2}|\leq C_2\xi r^{-\frac n2+2}=C_2u_V r\leq \frac12u_V,
\end{equation}
for any $r\leq 1/(2C_2)$. 

{Recall $A_1=kA_0$ with $k$ as \eqref{k-large}, now we set 
$$
\bar r_1=\min\left\{1,r_0,r'_0,\frac1{2C_2},\frac1{100C_0},\frac{1}{4(1+B+C_2+2C_1/n)(1+kC')}\right\},
$$ and take 
$$A_0=\max\left\{\frac{2C_1}{n},1,\frac{B+C_2}{\bar r_1}
\right\}.$$ Hence, we have \eqref{R-estimate}.} By the construction of $\xi_1$, \eqref{choose-1}, and \eqref{choose-2}, we have for any $r\leq \bar r_1$,
$$
\L w-\frac{n(n-2)}{4}w^{\frac{n+2}{n-2}}\leq -\frac n8A_0\xi^{\frac{n+2}{n-2}}r^{-\frac n2+1}<0.
$$
In addition, by \eqref{boundary-infty}, we have $w=+\infty$ on $\partial V\cap \{r\leq \bar r_1\}$. On the other hand, by \eqref{estimate-v-u_v}, \eqref{boundary-infty}, and the choice of $A_0$, we get, on $\{r=\bar r_1\}$,
$$
w-v\geq -B u_V \bar r_1- C_2u_V\bar r_1+A_0\bar r_1u_V\bar r_1>0.
$$
Therefore, similar to the proof of Theorem \ref{thm-main-1}, by applying the maximum principle and taking $\bar r_1$ smaller, we can prove 
$$
v\leq w=u_V+\xi_1 r^{-\frac n2+2}+A_0u_V r^2+A_1\phi_1 r^{-\frac n 2+3}\leq u_V+c_1 u_Vr +Cu_V d_g^2(x,0),
$$
for any $x\in \Omega\cap B_{r_1}(0)$, where $c_1=\xi_1/\xi\in C^{\infty}(\Sigma)\cap L^{\infty}(\Sigma)$ and $r_1=\bar r_1/\|T\|_{C^2}$.
Set $\bar w=u_V+\xi_1 r^{-\frac n2+2}-A_0u_V r^2-A_1\phi_1 r^{-\frac n 2+3}$. A similar discussion yields
$$
v\geq \bar w\geq u_V+c_1 u_Vr -Cu_V d_g^2(x,0),
$$
for any $x\in \Omega\cap B_{r_1}(0)$, by taking $r_1$ smaller. 

\textbf{Case 2. $\mu_1=2$. }

For some constants $A_0,A_1>0$ and some function $\xi_1\in C^\infty(\Sigma)$ {that satisfies \eqref{scalexi} to be determined, we set   
$$
w=u_V+\xi_1 r^{-\frac n2+2}+A_0u_V r^2-A_1\phi_1 r^{-\frac n 2+3}\log r=u_V(1+R_1), 
$$
where 
$$
R_1:=\frac{\xi_1}{\xi}r+A_0r^2-A_1\frac{\phi_1}{\xi}r^2\log r.
$$
We will choose $A_0$, $A_1$, $\xi_1$, and $\bar r_1\in(0,1/2)$ small enough so that
$$
C_0\bar r_1+A_0\bar r^2_1-A_1C'\bar r^2_1\log \bar r_1\leq \frac{1}{2},
$$
then, for any $r\leq \bar r_1$, \eqref{R1-exp} holds for such $R_1$.

By a computation similar to Case 1, we have
\begin{align*}
&\Delta w-\frac{1}{4}n(n-2)w^{\frac{n+2}{n-2}}\\
&=\Delta w-\frac{1}{4}n(n-2)u_V^{\frac{n+2}{n-2}}(1+R_1)^{\frac{n+2}{n-2}}\\
&\leq \Delta w-\frac{1}{4}n(n-2)u_V^{\frac{n+2}{n-2}}\Big[1+\frac{n+2}{n-2}R_1-C(n)R^2_1\Big]\\
&=\Big(\Delta_\theta \xi_1-\frac{n(n+2)}{4\rho^2}\xi_1+\frac{n(4-n)}{4}\xi_1\Big)r^{-\frac{n}{2}}\\
&\qquad-A_1\Big(\Big(-\frac{n}{2}+3\Big)\Big(\frac{n}{2}+1\Big)-\lambda_1\Big)\phi_1 r^{-\frac{n}{2}+1}\log r\\
&\qquad+A_0(4\xi-n\xi^{\frac{n+2}{n-2}})r^{-\frac{n}{2}+1}-4A_1\phi_1r^{-\frac n2+1}+\frac{1}{4}n(n-2)C(n)u_V^{\frac{n+2}{n-2}}R^2_1.
\end{align*}}
Note that $\mu_1=2$ implies that $(-n/2+3)(n/2+1)-\lambda_1=0$.
Hence, we have
\begin{align*}
&\L w-\frac{n(n-2)}{4}w^{\frac{n+2}{n-2}}\\
&=(\delta_{ij}+a_{ij,k}y_k+O(r^2)) w_{ij}+(b_{i,0}+O(r))w_i+cw\\
&=\Big(\Delta w-\frac{n(n-2)}{4}w^{\frac{n+2}{n-2}}\Big)+a_{ij,k}y_k w_{ij}+O(r^2)w_{ij}+b_{i,0}w_i+O(r)w_i+cw\\
&=\Big(\Delta w-\frac{n(n-2)}{4}w^{\frac{n+2}{n-2}}\Big)+a_{ij,k}y_k (u_{V})_{ij}+b_{i,0}(u_{V})_i\\
&\qquad+\xi^{\frac{n+2}{n-2}}O(r^{-\frac n2+1})+(A_0+A_1+1)\xi^{\frac{n+2}{n-2}}O(r^{-\frac n2+2}\log r)\\
&\leq\Big(\Delta_\theta \xi_1-\frac{n(n+2)}{4\rho^2}\xi_1+\frac{n(4-n)}{4}\xi_1-F_1\Big)r^{-\frac{n}{2}}\\
&\qquad+\big\{\xi^{\frac{n+2}{n-2}}O(1)+A_0(4\xi-n\xi^{\frac{n+2}{n-2}}) -4A_1\phi_1\big\}r^{-\frac n2+1}\\
&\qquad+(A_0+A_1+1)\xi^{\frac{n+2}{n-2}}O(r^{-\frac n2+2}\log r),
\end{align*}
{where $F_1$ is the same as that in Case 1.}
Proceeding as in Case 1, we can choose $\xi_1$, $A_0$, $A_1$, and $r_1$ such that
$$
v\leq w\leq u_V+c_1 u_Vr -Cu_V d_g^2(x,0)\log d_g(x,0),
$$
for any $x\in \Omega\cap B_{r_1}(0)$, where $c_1=\xi_1/\xi\in C^{\infty}(\Sigma)\cap L^{\infty}(\Sigma)$.
Set $\bar w=u_V+\xi_1 r^{-\frac n2+2}-A_0u_V r^2+A_1\phi_1 r^{-\frac n 2+3}\log r$. A similar discussion yields
$$
v\geq \bar w\geq u_V+c_1 u_Vr+Cu_V d_g^2(x,0)\log d_g(x,0),
$$
for any $x\in \Omega\cap B_{r_1}(0)$, by taking $r_1$ smaller. The above two inequalities imply the second part of \eqref{main-2}. 

Next, we improve this estimate to \eqref{main-2-3} if $T\in C^{3,\gamma}$, for some $\gamma\in(0,1)$. 
Note in this case, $a_{ij}\in C^{2,\gamma},b_i\in C^{1,\gamma}$, and we know
$$a_{ij}=\delta_{ij}+a_{ij,k}y_k+a_{ij,kl}y_ky_l+O(r^{2+\gamma}),\mbox{ and } b_i=b_{i, 0}+b_{i,k}y_k+O(r^{1+\gamma}).$$
 Let $A_0, A_1>0$ be constants and let $\xi_1, \xi_2,\xi_3\in C^\infty(\Sigma)$ be functions to be determined that satisfy estimates of the form \eqref{estimate-xi1} and \eqref{estimate-xi1-derivative}, i.e., for some $\beta\in (0,1)$,
\begin{equation}\label{scalexi_i}
|\xi_i|+d_\Sigma|\nabla_\theta\xi_i|+d_\Sigma^2|\nabla^2_\theta\xi_i|+d_\Sigma^{2+\beta}[\nabla^2_\theta \xi_i]_{C^{\beta}(B_{d_\Sigma(\cdot)/2}(\cdot))}\leq C_0\xi,\quad i=1,2,3,
\end{equation}
for some constants $C_0>0$ depends only on $n$, $\beta$, $g$, $V$, and the derivatives of $T$ at $0$.
We set   
\begin{align*}
w&=u_V+\xi_1 r^{-\frac n2+2}+\xi_2 r^{-\frac n2+3}\log r +\xi_3 r^{-\frac n2+3}+A_0u_V r^{2+\gamma}+A_1\phi_1r^{-\frac{n}{2}+1}(r^{2}-r^{2+\gamma})\\&=u_V(1+R_2), 
\end{align*}
where 
$$
R_2:=\frac{\xi_1}{\xi}r+\frac{\xi_2}{\xi}r^2\log r+\frac{\xi_3}{\xi}r^2+A_0r^{2+\gamma}+A_1\frac{\phi_1}{\xi}(r^2-r^{2+\gamma}).
$$
We will choose $A_0$, $A_1$, $\xi_i(i=1,2,3)$, and $\bar r_1\in(0,1/2)$ small enough so that
$$
C_0(\bar r_1-\bar r_1^2\log \bar r_1+\bar r^2_1)+A_0 \bar r^{2+\gamma}_1+A_1 C'\bar r_1^2\leq \frac12,
$$
hence we have, for any $r\leq \bar r_1$, 
$$
\left|(1+R_2)^{\frac{n+2}{n-2}}-1-\frac{n+2}{n-2}R_2-\frac{2(n+2)}{(n-2)^2}R^2_2\right|\leq C(n)|R_2|^3,
$$
where $C(n)>0$ depends only on $n$. 
Note that $\mu_1=2$ implies that $(-n/2+3)(n/2+1)-\lambda_1=0,$ then by $L_1\phi_1+\lambda_1\phi_1=0$, we have
$$
\Delta(\phi_1 r^{-\frac{n}{2}+3})-\frac{n(n+2)}{4\rho^2}r^{-\frac{n}{2}+1}\phi_1=0.
$$ A computation similar to Case 1 yields
\begin{align*}
&\Delta w-\frac{1}{4}n(n-2)w^{\frac{n+2}{n-2}}\\
&=\Delta w-\frac{1}{4}n(n-2)u_V^{\frac{n+2}{n-2}}(1+R_2)^{\frac{n+2}{n-2}}\\
&\leq \Delta w-\frac{1}{4}n(n-2)u_V^{\frac{n+2}{n-2}}\Big[1+\frac{n+2}{n-2}R_2+\frac{2(n+2)}{(n-2)^2}R^2_2-C(n)|R_2|^3\Big]\\
&=\Big(\Delta_\theta \xi_1-\frac{n(n+2)}{4\rho^2}\xi_1+\frac{n(4-n)}{4}\xi_1\Big)r^{-\frac{n}{2}}\\
&\qquad+\Big(\Delta_\theta \xi_2-\frac{n(n+2)}{4\rho^2}\xi_2+\left(-\frac{n}{2}+3 \right)\left(\frac{n}{2}+1\right)\xi_2\Big)r^{-\frac{n}{2}+1}\log r\\
&\qquad+\Big(\Delta_\theta \xi_3-\frac{n(n+2)}{4\rho^2}\xi_3+\left(-\frac{n}{2}+3\right)\left(\frac{n}{2}+1\right)\xi_3+4\xi_2-\frac{n(n+2)}{2(n-2)}\xi^{\frac{n+2}{n-2}}\Big(\frac{\xi_1}{\xi}\Big)^2\Big)r^{-\frac{n}{2}+1}\\
&\qquad+A_0\Big((2+\gamma)^2\xi-n\xi^{\frac{n+2}{n-2}}\Big)r^{-\frac{n}{2}+1+\gamma}-A_1\Big((2+\gamma)^2-2^2\Big)\phi_1r^{-\frac{n}{2}+1+\gamma}\\
&\qquad+(A_0+A_1+1)\xi^{\frac{n+2}{n-2}}O(r^{-\frac{n}{2}+2}).
\end{align*}
Hence, we have
\begin{align*}
&\L w-\frac{n(n-2)}{4}w^{\frac{n+2}{n-2}}\\
&=\Big(\Delta w-\frac{n(n-2)}{4}w^{\frac{n+2}{n-2}}\Big)+a_{ij,k}y_k w_{ij}+a_{ij,kl}y_ky_l w_{ij}+O(r^{2+\gamma})w_{ij}\\ &\qquad+b_{i,0}w_i+b_{i,k}y_kw_i+O(r^{1+\gamma})w_i+cw\\
&=\Big(\Delta w-\frac{n(n-2)}{4}w^{\frac{n+2}{n-2}}\Big)+a_{ij,k}y_k (u_{V})_{ij}+b_{i,0}(u_{V})_i\\&\qquad+a_{ij,kl}y_ky_l (u_V)_{ij}+b_{i,k}y_k(u_V)_i+cu_V+a_{ij,k}y_k(\xi_1r^{-\frac{n}{2}+2})_{ij}+b_{i,0}(\xi_1r^{-\frac{n}{2}+2})_{i}\\
&\qquad+O(r^{2+\gamma})(u_V)_{ij}+O(r^{1+\gamma})(u_V)_i+\xi^{\frac{n+2}{n-2}}O(r^{-\frac n2+2}\log r)\\&\qquad+(A_0+A_1+1)\xi^{\frac{n+2}{n-2}}O(r^{-\frac{n}{2}+2})\\
&\leq \Big(\Delta_\theta \xi_1-\frac{n(n+2)}{4\rho^2}\xi_1+\frac{n(4-n)}{4}\xi_1-F_1\Big)r^{-\frac{n}{2}}\\
&\qquad+\Big(\Delta_\theta \xi_2-\frac{n(n+2)}{4\rho^2}\xi_2+\left(-\frac{n}{2}+3 \right)\left(\frac{n}{2}+1\right)\xi_2\Big)r^{-\frac{n}{2}+1}\log r\\
&\qquad+\Big(\Delta_\theta \xi_3-\frac{n(n+2)}{4\rho^2}\xi_3+\left(-\frac{n}{2}+3 \right)\left(\frac{n}{2}+1\right)\xi_3+4\xi_2-F_2\Big)r^{-\frac{n}{2}+1}\\
&\qquad+\Big(\xi^{\frac{n+2}{n-2}}O(1)+A_0((2+\gamma)^2\xi-n\xi^{\frac{n+2}{n-2}})\Big)r^{-\frac n2+1+\gamma}-A_1\Big((2+\gamma)^2-2^2\Big)\phi_1r^{-\frac{n}{2}+1+\gamma}\\
&\qquad+(A_0+A_1+1)\xi^{\frac{n+2}{n-2}}O(r^{-\frac{n}{2}+2}),
\end{align*}
where \begin{align*}
F_1&=-r^{\frac n2}\Big(a_{ij,k}y_k (u_{V})_{ij}+b_{i,0}(u_{V})_i\Big),\\
F_2&=-r^{\frac n2-1}\Big\{a_{ij,kl}y_ky_l (u_V)_{ij}+b_{i,k}y_k(u_V)_i+cu_V\\
&\qquad+a_{ij,k}y_k(\xi_1r^{-\frac{n}{2}+2})_{ij}+b_{i,0}(\xi_1r^{-\frac{n}{2}+2})_{i}\Big\}+{\frac{n(n+2)}{2(n-2)}\xi^{\frac{n+2}{n-2}}\Big(\frac{\xi_1}{\xi}\Big)^2}.\end{align*}
First, by the same reasoning as in Case 1, we can find a $\xi_1\in C^\infty(\Sigma)$ such that 
\begin{align*}
    \Delta_\theta \xi_1-\frac{n(n+2)}{4\rho^2}\xi_1+\frac{n(4-n)}{4}\xi_1-F_1&=0.
\end{align*} 
Then we choose $\xi_2=-\frac{\sigma}{4} \phi_1 $, where $\sigma$ comes from Lemma \ref{solve-xi1}(ii) with $F=F_2$ and $k=2$ in \eqref{xi1-equation}. Finally, by Lemma \ref{solve-xi1}(ii) with $F=F_2$ and $k=2$, we can find $\xi_3\in C^\infty(\Sigma)$ such that 
\begin{align*}
    \Delta_\theta \xi_3-\frac{n(n+2)}{4\rho^2}\xi_3+\left(-\frac{n}{2}+3 \right)\left(\frac{n}{2}+1\right)\xi_3+4\xi_2-F_2&=0.
\end{align*}
By Lemma \ref{solve-xi1} and Lemma \ref{inner-problem-lemma}(4), we know that $\xi_i$ satisfies \eqref{scalexi_i}, $i=1,2,3$. Proceeding as in Case 1, we can choose $A_0$, $A_1$, and $r_1$ such that
$$
v\leq w\leq u_V+c_1 u_Vr+c_2u_V r^{2}\log r +Cu_V d^2_g(x, 0),
$$
for any $x\in \Omega\cap B_{r_1}(0)$, where $c_1=\xi_1/\xi$ and $c_2=\xi_2/\xi$ are both in $C^{\infty}(\Sigma)\cap L^{\infty}(\Sigma)$.
{Set $\bar w=u_V+\xi_1 r^{-\frac n2+2}+\xi_2 r^{-\frac n2+3}\log r +\xi_3 r^{-\frac n2+3}-A_0u_V r^{2+\gamma}-A_1\phi_1 r^{-\frac n 2+1}(r^2-r^{2+\gamma})$,} a similar discussion yields
$$
v\geq \bar w\geq u_V+c_1 u_Vr+c_2u_V r^{2}\log r -Cu_V d^2_g(x, 0),
$$
for any $x\in \Omega\cap B_{r_1}(0)$, by taking $r_1$ smaller.

\textbf{Case 3. $\mu_1<2$. }

For some constants $A_0,A_1>0$ and some function $\xi_1\in C^\infty(\Sigma)$ satisfying \eqref{scalexi} to be determined, we set{   
$$
w=u_V+\xi_1 r^{-\frac n2+2}+A_0u_V r^{2}+A_1\phi_1(r^{\mu_1-\frac{n-2}{2}}-r^{-\frac n 2+3})=u_V(1+R_1),
$$
where
$$
R_1:=\frac{\xi_1}{\xi}r+A_0r^2+A_1\frac{\phi_1}{\xi}(r^{\mu_1}-r^2).
$$
We will choose $A_0$, $A_1$, $\xi_1$, and $\bar r_1>0$ small enough so that
$$
C_0\bar r_1+A_0\bar r^2_1+A_1C'\bar r^{\mu_1}_1\leq\frac12,
$$
then, for any $r\leq \bar r_1$, \eqref{R1-exp} holds for such $R_1$.
}By $L_1\phi_1+\lambda_1\phi_1=0$ and \eqref{mu1-def}, we have
$$
\Delta(\phi_1 r^{\mu_1-\frac{n-2}{2}})-\frac{n(n+2)}{4\rho^2}r^{\mu_1-\frac{n+2}{2}}\phi_1=0.$$

By a computation similar to Case 1, we get
\begin{align*}
&\Delta w-\frac{1}{4}n(n-2)w^{\frac{n+2}{n-2}}\\
&=\Delta w-\frac{1}{4}n(n-2)u_V^{\frac{n+2}{n-2}}(1+R_1)^{\frac{n+2}{n-2}}\\
&\leq \Delta w-\frac{1}{4}n(n-2)u_V^{\frac{n+2}{n-2}}\Big[1+\frac{n+2}{n-2}R_1-C(n)R^2_1\Big]\\
&=\Big(\Delta_\theta \xi_1-\frac{n(n+2)}{4\rho^2}\xi_1+\frac{n(4-n)}{4}\xi_1\Big)r^{-\frac{n}{2}}+A_0(4\xi-n\xi^{\frac{n+2}{n-2}})r^{-\frac{n}{2}+1}\\
&\qquad-A_1\Big(\Big(-\frac{n}{2}+3\Big)\Big(\frac{n}{2}+1\Big)-\lambda_1\Big)\phi_1 r^{-\frac{n}{2}+1}+\frac{1}{4}n(n-2)C(n)u_V^{\frac{n+2}{n-2}}R^2_1.
\end{align*}
Note that $\mu_1<2$ implies that $(-n/2+3)(n/2+1)-\lambda_1>0$. By \eqref{mu1-def}, Lemma \ref{inner-problem-lemma}(1), and Lemma \ref{inner-problem-lemma}(3), we have $\mu_1>1$ for any $n\geq 3$. Hence, we obtain
\begin{align*}
&\L w-\frac{n(n-2)}{4}w^{\frac{n+2}{n-2}}\\
&=\Big(\Delta w-\frac{n(n-2)}{4}w^{\frac{n+2}{n-2}}\Big)+a_{ij,k}y_k w_{ij}+O(r^2)w_{ij}+b_{i,0}w_i+O(r)w_i+cw\\
&=\Big(\Delta w-\frac{n(n-2)}{4}w^{\frac{n+2}{n-2}}\Big)+a_{ij,k}y_k (u_{V})_{ij}+b_{i,0}(u_{V})_i\\
&\qquad+\xi^{\frac{n+2}{n-2}}O(r^{-\frac n2+1})+(A_0+A_1+1)\xi^{\frac{n+2}{n-2}}O(r^{-\frac n2+\mu_1})\\
&\leq\Big(\Delta_\theta \xi_1-\frac{n(n+2)}{4\rho^2}\xi_1+\frac{n(4-n)}{4}\xi_1-F_1\Big)r^{-\frac{n}{2}}\\
&\qquad+\big\{\xi^{\frac{n+2}{n-2}}O(1)+A_0(4\xi-n\xi^{\frac{n+2}{n-2}})-A_1\Big(\Big(-\frac{n}{2}+3\Big)\Big(\frac{n}{2}+1\Big)-\lambda_1\Big)\phi_1 \big\}r^{-\frac n2+1}\\
&\qquad+(A_0+A_1+1)\xi^{\frac{n+2}{n-2}}O(r^{-\frac n2+\mu_1}),
\end{align*}
{where $F_1$ is the same as in Case 1.}
Proceeding as in Case 1, we can choose $\xi_1$, $A_0$, $A_1$, and $r_1$ such that
$$
v\leq w\leq u_V+c_1 u_Vr +Cu_V d_g^{\mu_1}(x,0),
$$
for any $x\in \Omega\cap B_{r_1}(0)$, where $c_1=\xi_1/\xi\in C^{\infty}(\Sigma)\cap L^{\infty}(\Sigma)$.
Set $\bar w=u_V+\xi_1 r^{-\frac n2+2}-A_0u_V r^2-A_1\phi_1 (r^{\mu_1-\frac{n-2}{2}}-r^{-\frac n 2+3})$. A similar discussion yields
$$
v\geq \bar w\geq u_V+c_1 u_Vr-Cu_V d_g^{\mu_1}(x,0),
$$
for any $x\in \Omega\cap B_{r_1}(0)$, by taking $r_1$ smaller.
\end{proof}
\begin{remark}
    When $T$ is only $C^{2,\alpha}$ with some $\alpha\in(0, 1]$, we can still obtain the first order expansion of $u/(u_{V}\circ T)$. The proof strategies are similar. In fact, we can construct the related function $w$ as follows.
    \begin{enumerate}
        \item $w=u_V+\xi_1 r^{-\frac n2+2}+A_0u_V r^{1+\alpha}+A_1\phi_1 r^{-\frac n 2+2+\alpha}$,\mbox{ if } $\mu_1>1+\alpha$;
        \item $w=u_V+\xi_1 r^{-\frac n2+2}+A_0u_V r^{1+\alpha}-A_1\phi_1 r^{-\frac n 2+2+\alpha}\log r$,\mbox{ if } $\mu_1=1+\alpha$;
        \item $w=u_V+\xi_1 r^{-\frac n2+2}+A_0u_V r^{1+\alpha}+A_1\phi_1 (r^{\mu_1-\frac{n-2}{2}}-r^{-\frac n 2+2+\alpha})$, \mbox{ if } $\mu_1<1+\alpha$.
    \end{enumerate}
A similar analysis, combined with the inductive argument, leads to the result in Remark \ref{tmalpha}.
\end{remark}
\section{Examples with no tangential decay}\label{nontanexample}
In this section, we will introduce some examples in which the tangential decay is absent. Specifically, we will demonstrate some examples, where the estimate in Theorem \ref{thm-main-1} can not be improved as the following:
\begin{equation}\label{nontan}
    \left|\frac{u(x)}{u_V(Tx)}-1\right|\leq Cd_g(x, 0)o(1),
\end{equation}
for $x\in\Omega\cap B_{r_0}(0)$ and $\theta=Tx/{|Tx|}$, where $o(1)\to 0$ as $\rho(\theta)\to 0$ and $\rho\sim d_\Sigma$ in $\Sigma$.
This constitutes the key difference from the results in \cite{HJS}. In \cite{HJS}, where the domains are cones, the aforementioned estimate holds well and is instrumental in deriving higher-order expansions.

We begin by reviewing solutions of the Loewner-Nirenberg problem \eqref{equation-LN-1}-\eqref{boundary-condition-LN} in cones with translational invariance. As a preparatory step for our first example, we focus on the specific scenario with $n=3$, and the invariance is along the $z$-axis.

\begin{proposition}\label{examprop}
    Let $V_\alpha$ be a wedge in $\R^2$ with the origin as its vertex and with the opening angle $\alpha\in(0, \pi)$. Assume that $\partial V_\alpha=\{\tau=0\}\cup\{\tau=\alpha\}$, where $\tau\in(0, \pi)$ is defined by $(\cos\tau, \sin\tau)=\frac{1}{\sqrt{x^2+y^2}}(x,y)$. Suppose $V=V_\alpha\times\R:=\{(x, y, z): (x, y)\in V_\alpha,z\in\mathbb R\}\subset\R^3$, and $u_V(x, y, z)$ is the positive solution to the Loewner-Nirenberg problem \eqref{equation-LN-1}-\eqref{boundary-condition-LN} in $V$. Then
    $$u_V(x, y, z)=|\sqrt{x^2+y^2}|^{-\frac{1}{2}}\eta(\tau),$$
    where $\tau\in(0,\alpha)$, and $\eta\in C^\infty(0, \alpha)$ satisfies  
    \begin{equation}\label{examg}
\left\{
   \begin{aligned}
       &\eta''(\tau)+\frac{1}{4}\eta(\tau)=\frac{3}{4}\eta^5(\tau) \quad \mbox{ in } (0, \alpha), \\
       &\eta(0)=+\infty,\quad  \eta(\alpha)=+\infty.
   \end{aligned}
\right.
\end{equation}
\end{proposition}

The proof of this proposition can be found in \cite{HS}. Recall that in Section \ref{Intro} we also present another form of the solution of the Loewner–Nirenberg problem \eqref{equation-LN-1}–\eqref{boundary-condition-LN} in cones. By the uniqueness, this represents the same solution expressed in a different form.

\begin{example}\label{exampleA}
    Let $S_1$ and $S_2$ be two surfaces defined in $\R^3$ as follows: $$S_1=\left\{(x,y,z)\in\R^3: y=\frac{1+z}{100}x\right\}\mbox{ and } S_2=\left\{(x,y,z)\in\R^3: y=0\right\}.$$
    Let $\Omega$ be a domain bounded by $S_1$ and $S_2$, which is defined by $$\Omega=\left\{(x, y, z)\in B_1(0): 0\leq y\leq \frac{1+z}{100}x\right\}.$$Suppose $V_z$ is the tangent cone of $\Omega$ at $(0, 0, z)$, then $V_z$ is translationally invariant along the $z$-axis. For $i=1, 2$, let $P_i^z$ be the tangent plane of $S_i$ at $(0, 0, z).$ As in \cite{HS}, for $q\in V_z$, we define $f_{V_z}(d_1, d_2)$ by letting $f_{V_z}(d_{P_1^z}(q), d_{P_2^z}(q))=u_{V_z}(q)$ be the solution of the Loewner-Nirenberg problem \eqref{equation-LN-1}-\eqref{boundary-condition-LN} in $V_z$, where $d_{P_i^z}(q)$ is the distance from $q$ to $P^z_i$. 
    
    Suppose $u>0$ solves the Loewner-Nirenberg problem \eqref{equation-LN-1}-\eqref{boundary-condition-LN} in $\Omega$. Let $z>0$ and $p=(x, y, z)\in\Omega$ be a point such that the distances from $p$ to $S_1$ ans $S_2$ are $ks$ and $s$ respectively, where $s>0$ and $k\in(0, 1)$ are positive constants. When $p$ is close to the origin, by \cite{HS} we know 
    $$u(p)=f_{V_0}(ks, s)(1+O(r)),$$
    where $r=\sqrt{x^2+y^2+z^2}$ is the distance from $p$ to $0$. Similarly, when $p$  is close to $(0, 0, z)$,
    $$u(p)=f_{V_{z}}(ks, s)(1+O(\tilde{r})),$$ where $\tilde{r}=\sqrt{x^2+y^2}$ is the distance from $p$ to $(0, 0, z)$. Therefore, 
    \begin{equation} \label{exam1}
        \frac{u(p)}{f_{V_{0}}(ks, s)}=\frac{f_{V_{z}}(ks, s)}{f_{V_{0}}(ks, s)}(1+O(\tilde{r})).
    \end{equation}
By Proposition \ref{examprop}, we know $\frac{f_{V_{z}}(ks, s)}{f_{V_{0}}(ks, s)}$ is independent of $s$. We claim there are some $z>0$ and $k\in(0, 1)$ such that \begin{equation*}\label{claimexam}
\frac{f_{V_{z}}(ks, s)}{f_{V_{0}}(ks, s)}\neq 1.\end{equation*} 
Fixing such $z>0$ and $k\in(0, 1)$, and letting $s\to 0$ in both sides of \eqref{exam1} (note $\tilde{r}\to 0$ as $s\to 0$), we have
 \begin{equation} \label{exam3}
        \lim_{s\to 0}\frac{u(p)}{f_{V_{0}}(ks, s)}\neq 1.
    \end{equation}
We point out that $\rho(\theta_p)\to0$ as $s\to0$. If \eqref{nontan} holds, then $\lim\limits_{s\to 0}\dfrac{u(p)}{f_{V_{0}}(ks, s)}= 1,$
which contradicts \eqref{exam3}. The contradiction implies that \eqref{nontan} does not hold in this case.

Now we prove the claim by contradiction. We assume for any $z>0$ and $k\in(0, 1)$, ${f_{V_{z}}(ks, s)}={f_{V_{0}}(ks, s)}$. Then there exists a linear transformation $A_z$ such that $A_zV_0=V_z$, and $u_{V_0}(q)=u_{V_z}(A_zq)$ for $q\in V_0$. This is impossible since $A_z$ is not orthogonal.

\end{example}

Next, we present a more general method for constructing examples that exhibit no tangential decay. 
Surprisingly, even when the boundary is smooth and its tangent cone is a half-space, tangential decay is not guaranteed.
\begin{example}\label{exampleB}
For $x=(x_1, x_2, \cdots, x_n)\in\R^n$, let $r=|y|$ with $$y=Tx:=x+\left(x_1\sum_{k=1}^nx_k-\frac{1}{2}\sum_{k=1}^nx_k^2, x_2\sum_{k=1}^nx_k-\frac{1}{2}\sum_{k=1}^nx_k^2 ,\cdots,x_n\sum_{k=1}^nx_k-\frac{1}{2}\sum_{k=1}^nx_k^2\right).$$
Note $T\in C^\infty$ is invertible near $x=0$ and $T0=0$, $\nabla T(0)=\id$. Therefore, there is a constant $r_\star>0$ such that $T^{-1}$ is invertible in $B_{r_\star}(0)$. We define $S=T^{-1}$ and
$$\Omega=S(\R^n_+\cap B_{r_\star}(0)),$$ then $\R^n_+$ is the tangent cone of $\Omega$ at $x=0$. Equipping $\Omega$ with a metric $g$ as the setting in Section \ref{Intro}, we suppose $u$ is a local positive solution of the singular Yamabe problem \eqref{equation-Yamabe-main}-\eqref{boundary-condition-main} in $\Omega$. Then, by Theorem \ref{thm-main-2}, for $|x|$ small enough, we have
\begin{equation*}
\left|\frac{u(x)}{u_H(Tx)}-1-c_1\left(\frac{Tx}{|Tx|}\right)|Tx|\right|\leq
\begin{cases}
\begin{aligned}
 &Cd^2_g(x, 0),\text{ if }\mu_1>2,\\  
 &Cd^2_g(x, 0)|\log{d_g(x, 0)}|,\text{ if }\mu_1=2,\\
 &Cd^{\mu_1}_g(x, 0),\text{ if }\mu_1<2,
\end{aligned}
\end{cases} 
\end{equation*}
where $u_H=r^{-\frac{n-2}{2}}\xi$ is the solution of the Loewner-Nirenberg problem in $\R^n_+$, and $c_1(\theta)=\frac{\xi_1(\theta)}{\xi(\theta)}$. Here, $\xi$ is the solution of \eqref{equation-Yamabe-2}-\eqref{boundary-condition-2} in $\Sigma:=\partial B_1\cap\R_+^n$, and $\xi_1$ is a solution of \eqref{xi1-equation} in $\Sigma$ {with $k=1$, $\sigma=0$,}
\begin{align*}
&F(y)=-r^{\frac n 2}(a_{ij,k}y_k (u_{H})_{ij}+b_{i,0}(u_{H})_i)(y),\\
&a_{ij}=(g^{kl}T^i_kT^j_l)(Sy)=\delta_{ij}+a_{ij,k}y_k+O(r^2), \mbox{ as } r\to 0,\\
&b_i=\left(g^{kl}T_{kl}^i+\frac{1}{\sqrt{\mathrm{det}g}}\partial_{x_k}(\sqrt{\mathrm{det}g}g^{kl})T^i_l\right)(Sy)=b_{i,0}+O(r), \mbox{ as } r\to 0,\\
&g^{ij}=\delta_{ij}+O(r^2), \mbox{ as } r\to 0.
\end{align*}
By a direct calculation, we have
$$a_{ij,k}=(T^{i}_{jp}\partial_{y_k}(S^py)+T^{j}_{ip}\partial_{y_k}(S^py))(0)=2\delta_{ij},$$ and 
\begin{align*}
F&=-r^{\frac n 2}\left(2(y_1+\cdots+y_n)\Delta u_H+b_{i,0}(u_{H})_i\right)\\
&=-r^{\frac n 2}\left(\frac{n(n-2)}{2}(y_1+\cdots+y_n)(r^{-\frac{n+2}{2}}\xi^{\frac{n+2}{n-2}})+b_{i,0}(u_{H})_i\right).
\end{align*}
Let  $Q:=\partial B_1\cap\{y\in \R^n_+: y_1>0, y_2>0, \cdots, y_{n-1}>0\}\subset \Sigma$, then, $r^{\frac n 2}b_{i,0}(u_{H})_i\lesssim \xi^{\frac{n}{n-2}}$, and 
$(y_1+\cdots+y_n)r^{-1}\xi^{\frac{n+2}{n-2}}\sim\xi^{\frac{n+2}{n-2}}
$ in $Q$. Hence, we know $F\sim\xi^{\frac{n+2}{n-2}}$, and further $\rho^2F\sim \xi$ in $Q$. Note $\mu_1>1$ (by Remark \ref{mu1g1}), if \eqref{nontan} holds, then $\xi_1=o(1)\xi$ in $Q$, where $o(1)\to 0$ as $\rho(\theta)\to 0$. It is impossible when we compare the order of two sides of the equation \eqref{xi1-equation} in $Q$. Therefore, \eqref{nontan} dose not hold in this case.
\end{example}

\appendix
\section{Existence of Solutions to Degenerate/Singular Elliptic Equations}\label{app}
Let $\Sigma\subsetneq\mathbb{S}^{n-1}$ be a Lipschitz domain, $\xi\in C^{\infty}(\Sigma)$ be a positive
function in $\Sigma$ satisfying \eqref{equation-Yamabe-2}-\eqref{boundary-condition-2},
\begin{equation}
\rho=\xi^{-\frac{2}{n-2}},
\end{equation}
and $d_\Sigma=d_{g_{\mathbb{S}^{n-1}}}(\cdot,\partial\Sigma).$
By Lemma 2.4 in \cite{HJS}, we have $\rho\in C^{\infty}(\Sigma)\cap \mathrm{Lip}(\bar\Sigma)$, $\rho>0$ in $\Sigma$, $\rho=0$ on $\partial\Sigma$, and $\rho\sim d_{\Sigma}$.
Denote by $\{\theta_i\}^{n-1}_{i=1}$ normal local coordinates on the sphere $\mathbb{S}^{n-1}$.  Set $\partial_i=\partial_{\theta_i}$ and $\partial_{ij}=\partial_i\partial_j=\partial_j\partial_i$ for $i,j=1,\cdots,n-1$. We define the uniformly degenerate elliptic operator $L_0$ in $\Sigma$ by
\begin{equation}\label{L0} 
L_0=\rho^2a_{ij}\partial_{ij}+\rho b_i\partial_i +c,
\end{equation}
where $a_{ij},b_i,c\in L^{\infty}(\Sigma)$ satisfying $a_{ij}=a_{ji}$, and for any $\theta\in\Sigma$ and $\eta\in\mathbb R^{n-1}$,
\begin{equation}\label{elliptic}
\lambda|\eta|^2\leq a_{ij}(\theta)\eta_i\eta_j\leq \Lambda|\eta|^2,
\end{equation}
for some positive constants $\lambda\leq \Lambda$.


We now present an existence result for uniformly degenerate elliptic equations in Lipschitz domains, due to Graham and Lee \cite{GL}. A proof is provided here for the sake of completeness.

\begin{theorem}\label{outer-problem-theorem}
Let $\Sigma\subsetneq\mathbb{S}^{n-1}$ be a Lipschitz domain, $\xi\in C^{\infty}(\Sigma)$ be a positive
function in $\Sigma$ satisfying \eqref{equation-Yamabe-2}-\eqref{boundary-condition-2}, and $a_{ij},b_i,c\in L^{\infty}(\Sigma)\cap C^\alpha(\Sigma)$ with \eqref{elliptic}, for some $\alpha\in(0,1)$. Assume there exists a function $\psi\in C^2(\Sigma)$ such that $\psi>0$ in $\Sigma$, $\rho\nabla_\theta\log\psi\in L^\infty(\Sigma)$, and, for some constant $\delta>0$,
\begin{equation}\label{psi-super}
 L_0\psi\leq -\delta\psi\quad\text{in }\Sigma. 
\end{equation}
Then, for any $f\in C^\alpha(\Sigma)$ with $f/\psi\in L^\infty(\Sigma)$, there exists a solution $u\in C^{2,\alpha}(\Sigma)$ of 
\begin{equation}\label{L0-equation}
L_0u=f\quad\text{in }\Sigma,  
\end{equation}
and
\begin{equation}\label{u-psi-estimate}
\Big\|\frac u\psi\Big\|_{L^\infty(\Sigma)}\leq \frac1\delta\Big\|\frac f\psi\Big\|_{L^\infty(\Sigma)}.
\end{equation}
\end{theorem}
\begin{proof}
Consider an increasing sequence of domains $\{\Omega_k\}^\infty_{k=1}$ with boundaries of class $C^{2,\alpha}$ such that $\Omega_k\subset\subset\Omega_{k+1}\subset\subset\Sigma$ and $\Sigma=\cup^\infty_{k=1}\Omega_k$. Since $\rho>0$ in $\Sigma$,
the operator $L_0$ is uniformly elliptic in $\Omega_k$ with coefficients in $C^\alpha(\bar{\Omega}_k)$. By the Schauder theory, there exists a unique solution $u_k\in C^{2,\alpha}(\bar{\Omega}_k)$ of
\begin{align*}
    L_0u_k&=f\quad\text{in }\Omega_k , \\
    u_k&=0\quad\text{on } \partial\Omega_k .
\end{align*}
Define $v_k=u_k/\psi$. A direct calculation shows that $v_k$ satisfies
\begin{align*}
\rho^2a_{ij}\partial_{ij}v_k+\rho (b_i+2\rho a_{ij}\partial_j\log \psi)\partial_i v_k +\frac{L_0\psi}{\psi}v_k&=\frac{f}{\psi}\quad\text{in }\Omega_k,\\
v_k&=0\quad\text{on }\partial\Omega_k.
\end{align*}
By the maximum principle, we obtain, for any $k\geq 1$,
\begin{equation}\label{uk-psi-estimate}
\Big\|\frac {u_k}\psi\Big\|_{L^\infty(\Omega_k)}\leq \frac1\delta\Big\|\frac f\psi\Big\|_{L^\infty(\Omega_k)}\leq\frac1\delta\Big\|\frac f\psi\Big\|_{L^\infty(\Sigma)}.
\end{equation}   
Now fix an $i\geq1$ and take $k\geq i+1$. Then
$$
\|u_k\|_{L^{\infty}(\Omega_{i+1})}\leq \frac{1}{\delta}\|\psi\|_{L^{\infty}(\Omega_{i+1})}\Big\|\frac{f}{\psi}\Big\|_{L^\infty(\Sigma)}.
$$
By applying the interior Schauder estimates to $L_0u_k=f$ in $\Omega_i\subset\subset\Omega_{i+1}$, we get, for any $k\geq i+1$,
\begin{align*}
\|u_k\|_{C^{2,\alpha}(\bar{\Omega}_i)}&\leq C\{\|u_k\|_{L^\infty(\Omega_{i+1})}+\|f\|_{C^\alpha(\bar{\Omega}_{i+1})}\}\\
&\leq C\Big\{\frac{1}{\delta}\|\psi\|_{L^{\infty}(\Omega_{i+1})}\Big\|\frac{f}{\psi}\Big\|_{L^\infty(\Sigma)}+\|f\|_{C^\alpha(\bar{\Omega}_{i+1})}\Big\},
\end{align*}
where $C$ is a positive constant depending on $i$, but independent of $k$. Using a standard diagonal argument, we can extract a subsequence $\{u_{k'}\}$ of $\{u_k\}$ and find a function $u\in C^{2,\alpha}(\Sigma)$ such that $u_{k'}\rightarrow u$ in $C^{2}(\bar\Omega_i)$ for any $i\geq 1$. Thus, $L_0u=f$ in $\Sigma$. Passing to the limit $k\rightarrow\infty$ in \eqref{uk-psi-estimate} yields \eqref{u-psi-estimate}. 
\end{proof}
Next, we investigate the $L^2$-theory of a specific singular elliptic operator. We define 
\begin{equation}\label{L1}
L_1=\Delta_\theta-\frac{\kappa}{\rho^2}\quad\text{in }\Sigma,
\end{equation}
where $\kappa=n(n+2)/4$.

Here we collect some useful results from \cite{HJS,SW}. 
\begin{lemma}\label{inner-problem-lemma}
Let $\Sigma\subsetneq\mathbb{S}^{n-1}$ be a Lipschitz domain and let $\xi\in C^{\infty}(\Sigma)$ be a positive
function in $\Sigma$ satisfying \eqref{equation-Yamabe-2}-\eqref{boundary-condition-2}.

$\mathrm{(1)}$$($\cite[Theorem 4.2]{HJS}$)$ There exists an increasing sequence of positive constants $\{\lambda_i\}_{i\geq 1}$, diverging to $\infty$, and an $L^2(\Sigma)$-orthonormal basis $\{\phi_i\}_{i\geq 1}$ such that, for $i\geq 1$, $\phi_i\in C^{\infty}(\Sigma)\cap H^1_0(\Sigma)$ and $L_1\phi_i=-\lambda_i \phi_i$ holds in the weak sense.

$\mathrm{(2)}$$($Fredholm alternative, \cite[Theorem 4.3]{HJS}$)$ For any $\lambda\notin\{\lambda_i\}$ and any $f\in L^2(\Sigma)$, there exists a unique weak solution $u\in H^1_0(\Sigma)$ of 
\begin{equation}\label{L1-eigenvalue}
L_1 u+\lambda u=f\quad\text{in }\Sigma,
\end{equation}
satisfying the estimate
\begin{equation}\label{eigen-estimate}
\|u\|_{H^1_0(\Sigma)}\leq C\|f\|_{L^{2}(\Sigma)},
\end{equation}
where $C>0$ depends only on $n$, $\lambda$, and $\Sigma$.

For any $\lambda=\lambda_i$ for some $i$ and any $f\in L^2(\Sigma)$ with $(f,\phi_k)_{L^2(\Sigma)}=0$ for any
eigenfunction $\phi_k$ corresponding to the eigenvalue $\lambda_i$, there exists a unique weak solution $u\in H^1_0(\Sigma)$ of \eqref{L1-eigenvalue} such that $(u,\phi_k)_{L^2(\Sigma)}=0$ for any eigenfunction $\phi_k$ corresponding to the eigenvalue $\lambda_i$, and \eqref{eigen-estimate} holds.

$\mathrm{(3)}$$($\cite[Proposition 4.5]{SW}$)$ If $n=3$, then $\lambda_1>\frac{3}{4}$.

$\mathrm{(4)}$$($\cite[Theorem 4.4]{HJS}; \cite[ Lemma 4.1]{SW}$)$ Assume that for some $\lambda\in\mathbb{R}$ and $f\in C^{\infty}(\Sigma)$, $u\in H^1_0(\Sigma)$ is a weak solution of \eqref{L1-eigenvalue}. Then, there exists a constant $\nu>0$ depending only on $n$ and $\Sigma$ such that if, for some constant $A> 0$,
$$
|f|\leq A\rho^{\nu-2}\quad\text{in }\Sigma,
$$
then
\begin{equation}\label{inner-problem-infty-norm}
|u|\leq C(\|u\|_{L^2(\Sigma)}+A)\rho^\nu\quad\text{in }\Sigma,
\end{equation}
where $C>0$ depends only on $n$, $\lambda$, and $\Sigma$.
In particular, let $\phi_i$ be an eigenfunction as in $\mathrm{(1)}$. {Then, for any $\beta\in(0,1)$,
\begin{equation}\label{eigenvalue-function-estimate}
|\phi_i|+
\rho|\nabla_{\theta}\phi_i|+\rho^2|\nabla^2_{\theta}\phi_i|+\rho^{2+\beta}[\nabla^2_\theta \phi_i]_{C^{\beta}(B_{d_\Sigma(\cdot)/2}(\cdot))}\leq  C\rho^\nu\quad\text{in }\Sigma,
\end{equation}
where $C>0$ depends only on $n$, $\beta$, 
$\lambda_i$, and $\Sigma$.}
\end{lemma}

\bibliography{bib}

@incollection {LN,
    AUTHOR = {Loewner, Charles and Nirenberg, Louis},
     TITLE = {Partial differential equations invariant under conformal or
              projective transformations},
 BOOKTITLE = {Contributions to analysis (a collection of papers dedicated to
              {L}ipman {B}ers)},
     PAGES = {245--272},
 PUBLISHER = {Academic Press, New York-London},
      YEAR = {1974},
   MRCLASS = {35J65 (53C45)},
  MRNUMBER = {358078},
MRREVIEWER = {R. Osserman},
}

@article {AMc,
    AUTHOR = {Aviles, Patricio and McOwen, Robert C.},
     TITLE = {Complete conformal metrics with negative scalar curvature in
              compact {R}iemannian manifolds},
   JOURNAL = {Duke Math. J.},
  FJOURNAL = {Duke Mathematical Journal},
    VOLUME = {56},
      YEAR = {1988},
    NUMBER = {2},
     PAGES = {395--398},
      ISSN = {0012-7094},
   MRCLASS = {58G30 (35J60 53A30 53C21)},
  MRNUMBER = {932852},
MRREVIEWER = {Dennis M. DeTurck},
       DOI = {10.1215/S0012-7094-88-05616-5},
       URL = {https://doi.org/10.1215/S0012-7094-88-05616-5},
}

@article {CGS1989,
    AUTHOR = {Caffarelli, Luis A. and Gidas, Basilis and Spruck, Joel},
     TITLE = {Asymptotic symmetry and local behavior of semilinear elliptic
              equations with critical {S}obolev growth},
   JOURNAL = {Comm. Pure Appl. Math.},
  FJOURNAL = {Communications on Pure and Applied Mathematics},
    VOLUME = {42},
      YEAR = {1989},
    NUMBER = {3},
     PAGES = {271--297},
      ISSN = {0010-3640,1097-0312},
   MRCLASS = {35J60 (35B40 58G30)},
  MRNUMBER = {982351},
MRREVIEWER = {Robert\ McOwen},
       DOI = {10.1002/cpa.3160420304},
       URL = {https://doi.org/10.1002/cpa.3160420304},
}

@article {GL,
    AUTHOR = {Graham, C. Robin and Lee, John M.},
     TITLE = {Einstein metrics with prescribed conformal infinity on the
              ball},
   JOURNAL = {Adv. Math.},
  FJOURNAL = {Advances in Mathematics},
    VOLUME = {87},
      YEAR = {1991},
    NUMBER = {2},
     PAGES = {186--225},
      ISSN = {0001-8708},
   MRCLASS = {53C25 (53C50 58E11 58G20)},
  MRNUMBER = {1112625},
MRREVIEWER = {Dennis M. DeTurck},
       DOI = {10.1016/0001-8708(91)90071-E},
       URL = {https://doi.org/10.1016/0001-8708(91)90071-E},
}

@article {M,
    AUTHOR = {Mazzeo, Rafe},
     TITLE = {Regularity for the singular {Y}amabe problem},
   JOURNAL = {Indiana Univ. Math. J.},
  FJOURNAL = {Indiana University Mathematics Journal},
    VOLUME = {40},
      YEAR = {1991},
    NUMBER = {4},
     PAGES = {1277--1299},
      ISSN = {0022-2518,1943-5258},
   MRCLASS = {53C21 (35B40 35J60 35J70)},
  MRNUMBER = {1142715},
       DOI = {10.1512/iumj.1991.40.40057},
       URL = {https://doi.org/10.1512/iumj.1991.40.40057},
}

@article {ACF,
    AUTHOR = {Andersson, Lars and Chru\'sciel, Piotr T. and Friedrich,
              Helmut},
     TITLE = {On the regularity of solutions to the {Y}amabe equation and
              the existence of smooth hyperboloidal initial data for
              {E}instein's field equations},
   JOURNAL = {Comm. Math. Phys.},
  FJOURNAL = {Communications in Mathematical Physics},
    VOLUME = {149},
      YEAR = {1992},
    NUMBER = {3},
     PAGES = {587--612},
      ISSN = {0010-3616,1432-0916},
   MRCLASS = {53C21 (58G16 58G30 83C05)},
  MRNUMBER = {1186044},
MRREVIEWER = {Robert\ McOwen},
       URL = {http://projecteuclid.org/euclid.cmp/1104251309},
}

@article {KMPS1999,
    AUTHOR = {Korevaar, Nick and Mazzeo, Rafe and Pacard, Frank and Schoen, Richard},
     TITLE = {Refined asymptotics for constant scalar curvature metrics with
              isolated singularities},
   JOURNAL = {Invent. Math.},
  FJOURNAL = {Inventiones Mathematicae},
    VOLUME = {135},
      YEAR = {1999},
    NUMBER = {2},
     PAGES = {233--272},
      ISSN = {0020-9910,1432-1297},
   MRCLASS = {35J60 (35B05 53C21 58J60)},
  MRNUMBER = {1666838},
MRREVIEWER = {Ricardo\ Sa Earp},
       DOI = {10.1007/s002220050285},
       URL = {https://doi.org/10.1007/s002220050285},
}

@article {K,
    AUTHOR = {Kichenassamy, Satyanad},
     TITLE = {Boundary behavior in the {L}oewner-{N}irenberg problem},
   JOURNAL = {J. Funct. Anal.},
  FJOURNAL = {Journal of Functional Analysis},
    VOLUME = {222},
      YEAR = {2005},
    NUMBER = {1},
     PAGES = {98--113},
      ISSN = {0022-1236,1096-0783},
   MRCLASS = {35J60 (35B40 35B65 35C05)},
  MRNUMBER = {2129766},
MRREVIEWER = {Antonia\ Passarelli di Napoli},
       DOI = {10.1016/j.jfa.2004.06.014},
       URL = {https://doi.org/10.1016/j.jfa.2004.06.014},
}

@article {G,
    AUTHOR = {Guan, Bo},
     TITLE = {Complete conformal metrics of negative {R}icci curvature on
              compact manifolds with boundary},
   JOURNAL = {Int. Math. Res. Not. IMRN},
  FJOURNAL = {International Mathematics Research Notices. IMRN},
      YEAR = {2008},
     PAGES = {Art. ID rnn 105, 25},
      ISSN = {1073-7928,1687-0247},
   MRCLASS = {53C21},
  MRNUMBER = {2439538},
MRREVIEWER = {Harish\ Seshadri},
       DOI = {10.1093/imrn/rnn105},
       URL = {https://doi.org/10.1093/imrn/rnn105},
}

@article {M2008,
    AUTHOR = {Marques, Fernando C.},
     TITLE = {Isolated singularities of solutions to the {Y}amabe equation},
   JOURNAL = {Calc. Var. Partial Differential Equations},
  FJOURNAL = {Calculus of Variations and Partial Differential Equations},
    VOLUME = {32},
      YEAR = {2008},
    NUMBER = {3},
     PAGES = {349--371},
      ISSN = {0944-2669,1432-0835},
   MRCLASS = {35J60 (35A20 53C21)},
  MRNUMBER = {2393072},
MRREVIEWER = {Fernando\ A.\ Schwartz},
       DOI = {10.1007/s00526-007-0144-3},
       URL = {https://doi.org/10.1007/s00526-007-0144-3},
}

@article {Graham2017,
    AUTHOR = {Graham, C. Robin},
     TITLE = {Volume renormalization for singular {Y}amabe metrics},
   JOURNAL = {Proc. Amer. Math. Soc.},
  FJOURNAL = {Proceedings of the American Mathematical Society},
    VOLUME = {145},
      YEAR = {2017},
    NUMBER = {4},
     PAGES = {1781--1792},
      ISSN = {0002-9939,1088-6826},
   MRCLASS = {53A30 (53A55)},
  MRNUMBER = {3601568},
MRREVIEWER = {Juan\ Miguel\ Ruiz},
       DOI = {10.1090/proc/13530},
       URL = {https://doi.org/10.1090/proc/13530},
}

@article {GH2017,
    AUTHOR = {Gursky, Matthew J. and Han, Qing},
     TITLE = {Non-existence of {P}oincar\'e-{E}instein manifolds with
              prescribed conformal infinity},
   JOURNAL = {Geom. Funct. Anal.},
  FJOURNAL = {Geometric and Functional Analysis},
    VOLUME = {27},
      YEAR = {2017},
    NUMBER = {4},
     PAGES = {863--879},
      ISSN = {1016-443X,1420-8970},
   MRCLASS = {53C20 (53A30 53C25)},
  MRNUMBER = {3678503},
MRREVIEWER = {Gordon\ Craig},
       DOI = {10.1007/s00039-017-0414-y},
       URL = {https://doi.org/10.1007/s00039-017-0414-y},
}

@article {CLW2019,
    AUTHOR = {Chen, Xuezhang and Lai, Mijia and Wang, Fang},
     TITLE = {Escobar-{Y}amabe compactifications for {P}oincar\'e-{E}instein
              manifolds and rigidity theorems},
   JOURNAL = {Adv. Math.},
  FJOURNAL = {Advances in Mathematics},
    VOLUME = {343},
      YEAR = {2019},
     PAGES = {16--35},
      ISSN = {0001-8708,1090-2082},
   MRCLASS = {53A30 (53C20 53C24)},
  MRNUMBER = {3880687},
       DOI = {10.1016/j.aim.2018.11.005},
       URL = {https://doi.org/10.1016/j.aim.2018.11.005},
}

@article {HS,
    AUTHOR = {Han, Qing and Shen, Weiming},
     TITLE = {The {L}oewner-{N}irenberg problem in singular domains},
   JOURNAL = {J. Funct. Anal.},
  FJOURNAL = {Journal of Functional Analysis},
    VOLUME = {279},
      YEAR = {2020},
    NUMBER = {6},
     PAGES = {108604, 43},
      ISSN = {0022-1236},
   MRCLASS = {35J91 (35C20)},
  MRNUMBER = {4096721},
MRREVIEWER = {Abdolrahman Razani},
       DOI = {10.1016/j.jfa.2020.108604},
       URL = {https://doi.org/10.1016/j.jfa.2020.108604},
}

@article {GG2021,
    AUTHOR = {Graham, C. Robin and Gursky, Matthew J.},
     TITLE = {Chern-{G}auss-{B}onnet formula for singular {Y}amabe metrics
              in dimension four},
   JOURNAL = {Indiana Univ. Math. J.},
  FJOURNAL = {Indiana University Mathematics Journal},
    VOLUME = {70},
      YEAR = {2021},
    NUMBER = {3},
     PAGES = {1131--1166},
      ISSN = {0022-2518,1943-5258},
   MRCLASS = {53C18 (53C20)},
  MRNUMBER = {4284109},
MRREVIEWER = {Yongbing\ Zhang},
       DOI = {10.1512/iumj.2021.70.8491},
       URL = {https://doi.org/10.1512/iumj.2021.70.8491},
}

@article {HLL2021,
    AUTHOR = {Han, Qing and Li, Xiaoxiao and Li, Yichao},
     TITLE = {Asymptotic expansions of solutions of the {Y}amabe equation
              and the {$\sigma_k$}-{Y}amabe equation near isolated singular
              points},
   JOURNAL = {Comm. Pure Appl. Math.},
  FJOURNAL = {Communications on Pure and Applied Mathematics},
    VOLUME = {74},
      YEAR = {2021},
    NUMBER = {9},
     PAGES = {1915--1970},
      ISSN = {0010-3640,1097-0312},
   MRCLASS = {58J05 (53C18)},
  MRNUMBER = {4287691},
MRREVIEWER = {Qu\cfac oc\ Anh\ Ng\^o},
       DOI = {10.1002/cpa.21943},
       URL = {https://doi.org/10.1002/cpa.21943},
}

@article {Jiang,
    AUTHOR = {Jiang, Xumin},
     TITLE = {Boundary expansion for the {L}oewner-{N}irenberg problem in
              domains with conic singularities},
   JOURNAL = {J. Funct. Anal.},
  FJOURNAL = {Journal of Functional Analysis},
    VOLUME = {281},
      YEAR = {2021},
    NUMBER = {7},
     PAGES = {Paper No. 109122, 41},
      ISSN = {0022-1236},
   MRCLASS = {35J60 (58J05)},
  MRNUMBER = {4269604},
MRREVIEWER = {Yuhua Sun},
       DOI = {10.1016/j.jfa.2021.109122},
       URL = {https://doi.org/10.1016/j.jfa.2021.109122},
}

@article {CMY2022,
    AUTHOR = {Chang, Sun-Yung Alice and McKeown, Stephen E. and Yang, Paul},
     TITLE = {Scattering on singular {Y}amabe spaces},
   JOURNAL = {Rev. Mat. Iberoam.},
  FJOURNAL = {Revista Matem\'atica Iberoamericana},
    VOLUME = {38},
      YEAR = {2022},
    NUMBER = {7},
     PAGES = {2153--2184},
      ISSN = {0213-2230,2235-0616},
   MRCLASS = {53C18 (53C40 58J50)},
  MRNUMBER = {4526312},
MRREVIEWER = {Yongbing\ Zhang},
       DOI = {10.4171/rmi/1390},
       URL = {https://doi.org/10.4171/rmi/1390},
}

@article {XZ2022,
    AUTHOR = {Xiong, J. and Zhang, L.},
     TITLE = {Isolated singularities of solutions to the {Y}amabe equation
              in dimension 6},
   JOURNAL = {Int. Math. Res. Not. IMRN},
  FJOURNAL = {International Mathematics Research Notices. IMRN},
      YEAR = {2022},
    NUMBER = {12},
     PAGES = {9571--9597},
      ISSN = {1073-7928,1687-0247},
   MRCLASS = {35J65},
  MRNUMBER = {4436213},
       DOI = {10.1093/imrn/rnab090},
       URL = {https://doi.org/10.1093/imrn/rnab090},
}

@article {HXZ2023,
    AUTHOR = {Han, Zheng-Chao and Xiong, Jingang and Zhang, Lei},
     TITLE = {Asymptotic behavior of solutions to the {Y}amabe equation with
              an asymptotically flat metric},
   JOURNAL = {J. Funct. Anal.},
  FJOURNAL = {Journal of Functional Analysis},
    VOLUME = {285},
      YEAR = {2023},
    NUMBER = {4},
     PAGES = {Paper No. 109982, 55},
      ISSN = {0022-1236,1096-0783},
   MRCLASS = {35Q75 (53C20 58J37)},
  MRNUMBER = {4584986},
       DOI = {10.1016/j.jfa.2023.109982},
       URL = {https://doi.org/10.1016/j.jfa.2023.109982},
}

@article {HJS,
    AUTHOR = {Han, Qing and Jiang, Xumin and Shen, Weiming},
     TITLE = {The {L}oewner-{N}irenberg problem in cones},
   JOURNAL = {J. Funct. Anal.},
  FJOURNAL = {Journal of Functional Analysis},
    VOLUME = {287},
      YEAR = {2024},
    NUMBER = {8},
     PAGES = {110566, 54},
      ISSN = {0022-1236},
   MRCLASS = {35J75 (35B40 35C20 58J37)},
  MRNUMBER = {4771001},
       DOI = {10.1016/j.jfa.2024.110566},
       URL = {https://doi.org/10.1016/j.jfa.2024.110566},
}

@article {SW,
    AUTHOR = {Shen, Weiming and Wang, Yue},
     TITLE = {Asymptotic behavior of complete conformal metric near singular
              boundary},
   JOURNAL = {Adv. Math.},
  FJOURNAL = {Advances in Mathematics},
    VOLUME = {458},
      YEAR = {2024},
    NUMBER = {part B},
     PAGES = {Paper No. 109977, 34},
      ISSN = {0001-8708},
   MRCLASS = {35B40 (35C20 35J15 35R01 53C18 58J37)},
  MRNUMBER = {4812245},
MRREVIEWER = {A. Dehghan Nezhad},
       DOI = {10.1016/j.aim.2024.109977},
       URL = {https://doi.org/10.1016/j.aim.2024.109977},
}

@article {LG,
    AUTHOR = {Li, Gang},
     TITLE = {Two flow approaches to the {L}oewner-{N}irenberg problem on
              manifolds},
   JOURNAL = {J. Geom. Anal.},
  FJOURNAL = {Journal of Geometric Analysis},
    VOLUME = {32},
      YEAR = {2022},
    NUMBER = {1},
     PAGES = {Paper No. 7, 30},
      ISSN = {1050-6926,1559-002X},
   MRCLASS = {53E99 (35K55 35R01 53C18 53C21)},
  MRNUMBER = {4346774},
MRREVIEWER = {Yongbing\ Zhang},
       DOI = {10.1007/s12220-021-00800-3},
       URL = {https://doi.org/10.1007/s12220-021-00800-3},
}
\bibliographystyle{plain}
\end{document}